\documentclass[11pt,a4paper,reqno]{amsart}
\usepackage[hmargin=2.5cm,bmargin=2.5cm,tmargin=3cm]{geometry}
\usepackage{enumerate}
\usepackage{amsmath,amsthm,amssymb,amsfonts,latexsym}
\usepackage{mathbbol}

\usepackage[colorlinks, linkcolor=blue, filecolor=blue,
     citecolor = olive, urlcolor=purple]{hyperref}
\usepackage[gen]{eurosym}
\usepackage{aligned-overset}
\usepackage[german,english]{babel}
\usepackage[normalem]{ulem}
\usepackage{latexsym}

\usepackage{tikz}					

\DeclareMathOperator{\Id}{Id}
\DeclareMathOperator{\id}{id}

\DeclareMathOperator*{\esssup}{{ess\;sup}}

\DeclareMathOperator{\dist}{dist}

\DeclareMathOperator{\real}{Re}
\DeclareMathOperator{\essinf}{ess\; inf}
\allowdisplaybreaks
\newcommand{\icomp}{\mathrm{i}}
\newcommand{\e}{\mathrm{e}}
\newcommand{\Graph}{{\mathcal G}}

\newcommand{\Din}{C^{0,\mathrm{Dini}}}

\newcommand{\CC}{\mathbb{C}}

\newcommand{\KK}{\mathbb{K}}
\newcommand{\NN}{\mathbb{N}}

\newcommand{\RR}{\mathbb{R}}

\newcommand{\dX}{\mathrm{d}_X}

\newcommand{\diffd}{\mathrm{d}} 
\newcommand{\dx}[1][x]{\,\diffd#1}

\usepackage{aliascnt}

\theoremstyle{plain}
\newtheorem{theo}{Theorem}[section]

\newaliascnt{cor}{theo}
\newaliascnt{prop}{theo}
\newaliascnt{lemma}{theo}

\newtheorem{lemma}[lemma]{Lemma}
\newtheorem{prop}[prop]{Proposition}
\newtheorem{cor}[cor]{Corollary}

\aliascntresetthe{cor}
\aliascntresetthe{prop}
\aliascntresetthe{lemma}

\theoremstyle{defi} 

\newaliascnt{defi}{theo}
\newaliascnt{assum}{theo}
\newaliascnt{assums}{theo}
\newaliascnt{prob}{theo}

\newtheorem{defi}[defi]{Definition}

\aliascntresetthe{defi}
\aliascntresetthe{assum}
\aliascntresetthe{assums}
\aliascntresetthe{prob}

\theoremstyle{rem}

\newaliascnt{rems}{theo}
\newaliascnt{rem}{theo}
\newaliascnt{exa}{theo}
\newaliascnt{exs}{theo}

\newtheorem{rem}[rem]{Remark}

\aliascntresetthe{rems}
\aliascntresetthe{rem}
\aliascntresetthe{exa}
\aliascntresetthe{exs}

\numberwithin{equation}{section}
\numberwithin{lemma}{section}

\newcommand{\be}{\begin{equation}}
\newcommand{\ee}{\end{equation}}
\newcommand{\bea}{\begin{eqnarray*}}
	\newcommand{\eea}{\end{eqnarray*}}
\newcommand{\beq}{\begin{eqnarray}}
\newcommand{\eeq}{\end{eqnarray}}

\newcommand{\mV}{\mathsf V}
\newcommand{\mVD}{{\mV_{\mathsf{D}}}}

\newcommand{\mv}{\mathsf v}
\newcommand{\mE}{\mathsf E}
\newcommand{\me}{\mathsf e}

\usepackage{todonotes}

\title{On the Lipschitz continuity of the heat kernel} 

\author[P.~Bifulco]{Patrizio Bifulco}
\author[D.~Mugnolo]{Delio Mugnolo}

\address{Lehrgebiet Analysis, Fakult\"at Mathematik und Informatik, Fern\-Universit\"at in Hagen, D-58084 Hagen, Germany}
\email{patrizio.bifulco@fernuni-hagen.de}

\email{delio.mugnolo@fernuni-hagen.de}

\thanks{We would like to thank Tom ter Elst (Auckland) for pointing us at \cite{AusMcITch98,ElsReh15,ElsOuh19,ElsWon20} and for inspiring discussions. We would also like to thank Uta Freiberg (Chemnitz) for pointing us at \cite{Str06} and for helpful comments; and to Jochen Glück (Wuppertal) for suggesting a partially alternative approach to the proof of \autoref{lem:lipschitzonecoordinate}, which eventually led to \autoref{rem:heat-kernel-buc-second-variable}.  The first author warmly thanks Joachim Kerner (Hagen) for interesting exchange. Both authors are grateful to the anonymous referee for their careful reading and their insightful remarks that improved the quality of the manuscript. \\ Both authors were supported by the Deutsche Forschungsgemeinschaft DFG (Grant 397230547).
}

\begin{document}
	
	\begin{abstract} 
We study integral kernels of strongly continuous semigroups on Lebesgue spaces over metric measure spaces. Based on semigroup smoothing properties and abstract Morrey-type inequalities, we give sufficient conditions for Hölder or Lipschitz continuity of the kernels. 
We apply our results to (pseudo)differential operators on domains and quantum graphs, to Laplacians on a class of fractals including the Sierpiński gasket, and to structurally damped wave equations. An extension to non-autonomous problems is also discussed.
	\end{abstract}
\maketitle



 \section{Introduction}
The celebrated Mercer Theorem is a powerful tool to describe the kernel $k$ of an integral operator $T_k$ acting on  $L^2(X)$, where $X$ is a suitable metric measure space. In particular, it yields that $k$ is a (uniformly convergent) series involving the eigenvalues and eigenvectors of $T_k$. 
A classical application of the Mercer Theorem deals with heat kernels, i.e., the integral kernels of strongly continuous semigroups (if such kernel exists).  In the case of Laplacians with Dirichlet or Neumann conditions on bounded domains with fairly rough boundary, the joint continuity of the heat kernel up to boundary has been discussed in~\cite[Section~6]{AreEls19}; earlier results go back at least to~\cite[Theorem~A.1.3]{Sim82}.
 In many applications, though, the eigenvectors of the semigroup's generator have additional smoothness properties -- say, they are Hölder continuous -- and it is thus natural to wonder whether this smoothness is inherited by the heat kernel. 

A partial answer to this question is based on the Schwartz Kernel Theorem and its outgrowths: kernel smoothness with respect to the spatial components is then known for all operators mapping  distributions into smooth functions, see~\cite[Theorem~5.2.6]{Hor90}: this is especially the case for the Laplacian on Euclidean domains with smooth boundary, since the semigroup it generates is analytic -- hence it maps all initial data into the domain $D(A^\infty)$ of arbitrarily high powers of the generator $A$ -- and because $D(A^\infty)$ can actually be characterized in terms of smooth functions, provided boundary regularity results are available. Similar results are known for the heat kernel associated with the Friedrichs extension of the Laplace--Beltrami operator on Riemannian manifolds, cf.~\cite[Theorem~5.2.1]{Dav89}.

One may then study weaker smoothness properties for more general classes of operators.  The present article aims at proposing abstract conditions for a semigroup acting on $L^r(X)$, where $X$ is a general metric measure space,
to have an integral kernel that is $\alpha$-Hölder continuous in both coordinates,  or even {jointly $\alpha$-Hölder} continuous. Remarkably, less attention has been seemingly devoted to the extreme case of $\alpha=1$, i.e., to Lipschitz continuity.


The $\alpha$-Hölder continuity, $\alpha \in (0,1)$, in each coordinate for heat kernels of bounded analytic semigroups  was first studied in~\cite{Ouh98}: Ouhabaz showed that this is equivalent to certain Gagliardo--Nirenberg-type inequalities along with ultracontractivity of the semigroup (in turn, this is known to be equivalent to suitable Sobolev-type embeddings),  see~\autoref{rem:additional-conds} below for more details.   Comparable results for volume doubling metric measure spaces have been obtained in \cite{GriTel12}.

The aim of this article is to show that these Gagliardo--Nirenberg-type inequalities can be replaced by abstract Morrey inequalities, i.e., the assumption that the domain $D(A^k)$ of some $k$-th power of the generator $A$ is continuously embedded in the space of Hölder continuous functions:
in common applications to (pseudo)differential operators on Euclidean domains, this condition is arguably easier to check, see \autoref{sec:applications}. Indeed, related ideas have already been explored in~\cite{Cou03},  under the fairly strong assumption that $(\e^{tA})_{t\ge 0}$ satisfies upper \textit{and} lower Gaussian-type kernel estimates.
In one dimension, for rather general operators with complex-valued coefficients
on $L^2(\RR)$, joint Lipschitz continuity of the heat kernel is already known, see~\cite[Theorem~2.36]{AusMcITch98}. 
More recently, joint Hölder continuity of the integral kernels associated with second order uniformly elliptic operators with $L^\infty$ coefficients on bounded open domains of $\RR^d$ has been proved -- under rather weak assumptions on the smoothness of the boundary -- in \cite[Theorem~1.3]{ElsReh15}.
Under additional smoothness assumptions on the elliptic coefficients, joint Hölder continuous differentiability was proved in~\cite[Theorem~3.1]{ElsOuh19}; see also \cite[Theorem~7.5]{ElsReh15},~\cite[Theorem~3.16]{ElsOuh19} and~\cite[Theorem~1.1]{ElsWon20} for the case of unbounded domains and complex lower order coefficients with Dirichlet or Robin conditions.

This question can be lifted to a more general class of investigations on heat kernels as follows: since the space of Lipschitz functions agrees with the Sobolev space $W^{1,\infty}$ for intervals or domains with smooth boundary, and more generally for quasiconvex domains (see \cite[Theorem~4.1]{Hei05}), one may also ask whether the heat kernel lies in $W^{1, r}$ for some ${ r}\in (2,\infty]$: this has been addressed since~\cite{AusCouDuo04}, see \cite{CouJiaKos20} and references therein for more recent results.

In this paper we elaborate on these and further related results. As mentioned above, in relevant cases we actually expect the heat kernel to enjoy even higher regularity. 
In \autoref{lem:lipschitzonecoordinate} 
we are going to propose sufficient conditions for Hölder or Lipschitz continuity of the heat kernel of semigroups $\e^{tA}$ in each coordinate, as well as for higher regularity: on general metric measure spaces $X$ we here use Hölder regularity of $A' p_t(\cdot,y)$
as an abstract proxy for the smoothness property $p_t(\cdot,y)\in C^k(\Omega)$ for Euclidean domains $\Omega$.
In our second main result,
\autoref{thm:abstractlipschitzmeasurespaces}, we provide sufficient conditions  for \emph{joint} Hölder continuity of the heat kernel.
 Our conditions are especially easy to check in the Hilbert space case, see~\autoref{cor:forms}.
In \autoref{sec:nonsmooth} and \autoref{sec:nonauto} we show the flexibility of our method by discussing Hölder, or even Lipschitz continuity of the heat kernel in two non-standard settings: heat flows on infinitesimally Hilbertian metric spaces and evolution families associated with non-autonomous sesquilinear forms, respectively.



In \autoref{sec:applications} we apply our results to several classes of operators.
Using our abstract results, we show in \autoref{subs:laplacians-domains} 
joint Lipschitz continuity up to the boundary of the heat kernel associated with uniformly elliptic operators with smooth coefficients on bounded domains of $\RR^d$: we regard this section as a sandbox to illustrate the basic principles underlying our theory, and do not have the ambition to obtain optimal regularity results under minimal  smoothness assumptions on the elliptic coefficients and the boundary. 
In \autoref{sec:metricgraphs} we discuss Laplacians (and a class of more general Schrödinger operators with both magnetic and electric potential) on possibly infinite metric graphs and prove that their heat kernel is joint Lipschitz continuous, while not being continuously differentiable even in one coordinate.
In \autoref{sec:fractals} we focus on a famous fractal, the Sierpiński gasket, and on a class of its generalizations. The Laplacian thereon is associated with a Dirichlet form, and indeed with a heat kernel: we prove, furthermore, Hölder continuity of this kernel with respect to the so-called \textit{resistance metric}.
Finally, in \autoref{sec:damped-wave} we prove joint Lipschitz continuity for the integral kernel of the semigroup that governs a one-dimensional damped wave equation.

 \section{The setup}
  Throughout this article, we consider a \textit{generalized metric measure space}, i.e., a triple
\[
(X,\dX,\mu)
\]
consisting of a set $X$, a generalized metric $\dX$ and a $\sigma$-finite measure $\mu$ on $X$ (by \textit{generalized metric} we mean a function from $X\times X$ to $[0,\infty]$ that satisfies the usual axioms of a metric). This setting includes -- but is not limited to -- the case of a locally compact (possibly disconnected: see \autoref{sec:damped-wave}) Hausdorff space $(X,\dX)$ endowed with a Borel measure $\mu$. 

Given a Banach space $E$, a function $f: X \rightarrow E$ is \emph{Hölder continuous of exponent $\alpha \in (0,1]$} (or simply \textit{$\alpha$-Hölder continuous}) if
 \[
| f|_{C^{0,\alpha}} := \sup_{x \neq x'} \frac{\Vert f(x)-f(x') \Vert_E}{\dX(x,x')^\alpha} < \infty:
 \]
as usual, the case $\alpha=1$ corresponds to \emph{Lipschitz continuous} functions.
We denote by $C^{0,\alpha}(X;E)$ the Banach space of bounded $\alpha$-Hölder continuous functions on $X$ with values in $E$, endowed with the norm
$\Vert \cdot \Vert_{C^{0,\alpha}}:= \Vert \cdot \Vert_\infty + | \cdot |_{C^{0,\alpha}}.$
In the case that $E=\mathbb{C}$ together with the Euclidean norm, we just write $C^{0,\alpha}(X)$ instead of $C^{0,\alpha}(X;\mathbb{C})$.

 
Also, given a generalized metric measure space $(X,\dX,\mu)$ one can also consider the product space $(X \times X, \mu \otimes \mu, \mathrm{d}_{X \times X})$ where $\mu \otimes \mu$ denotes the usual product measure on the product $\sigma$-algebra $\Sigma_X \otimes \Sigma_X$ on $X \times X$ and $\mathrm{d}_{X \times X}$ the canonical sum metric on $X\times X$.
One can hence consider the notion of Hölder continuity onto the product metric space $(X\times X,  \mathrm{d}_{X \times X})$. 
More precisely, we say that $f: X \times X \rightarrow E$ is \emph{jointly Hölder continuous of exponent $\alpha \in (0,1]$} 
 on $X \times X$ if $f$ is $\alpha$-Hölder 
 continuous with respect to the  sum metric $d_{X\times X}$; we denote by $C^{0,\alpha}(X \times X;E)$ the space of such functions.
 Additionally, $f: X \times X \rightarrow E$ is said to be \emph{Hölder continuous with exponent $\alpha$ in each coordinate} 
 if for every $x,y \in X$ the maps
 \[
 f_y: X\ni x\mapsto f(x,y)\in E \quad \text{and} \quad f_x: X\ni y\mapsto  f(x,y)
\in E 
\]
 both belong to $C^{0,\alpha}(X;E)$. 
 
Let $r \in [1,\infty)$ and $r' := \frac{r}{r-1}$.
 A linear operator $T: L^r(X) \rightarrow L^{r'}(X)$ is called an \emph{integral operator} if there exists a measurable function $k: X \times X \rightarrow \mathbb{C}$ -- the \emph{integral kernel} of $T$ -- such that $k(x,\cdot)f \in L^1(X)$ 
and
 \begin{equation}\label{eq:integral-op}
 (Tf)(x) = \int_X k(x,y)f(y) \dx[\mu(y)]\qquad \hbox{for all }f\in L^r(X)\hbox{ and a.e.\ }x\in X.
 \end{equation}

At the risk of redundancy, let us stress that many operators satisfy \eqref{eq:integral-op} for some $k$ that is \textit{not} an integral kernel in the above sense: for instance, the identity can be represented as in~\eqref{eq:integral-op} for the singular kernel given by $k(x,y):=\delta_x(y)$.
 
Let us recall the following fact, which explains in a clearer way why the above notion of integral operator is rather restrictive. It  will be crucial in our analysis: we refer to the  proof of \cite[Theorem 1.3]{AreBuk94}.
\begin{lemma}\label{lem:bukhv} 
For any $r \in [1,\infty)$ and any $\sigma$-finite measure space $(X,\mu)$, there is an isometric isomorphism between  $L^\infty(X;L^{r'}(X))$ and $\mathcal{L}(L^r(X);L^\infty(X))$.
\end{lemma}
  
In particular, each bounded linear operator from $L^r(X)$ to $L^\infty(X)$ is an integral operator with integral kernel belonging to $L^\infty(X;L^{r'}(X))$, and vice versa. If a semigroup of bounded linear operators consists of integral operators in the above sense, the family of their integral kernels is commonly referred to as the semigroup's \textit{heat kernel}: this notion has commonly been used in the literature in connection with strongly continuous, symmetric Markovian semigroups -- see, e.g., the fairly abstract metric measure theoretical treatment in~\cite{PosRuc18} -- but has also been extended to further classes, including strongly continuous semigroups on $L^2(\Omega)$ generated by elliptic operators with complex coefficients on $\Omega\subset \RR^d$ \cite{AusMcITch98,Ouh05}. At any rate, we are only aware of investigations that take the cue from a semigroup on an $L^2$ space.

We are here going to propose an abstract notion of heat kernel that depends on the tangible properties of a generator, rather than on those of the -- usually elusive -- semigroup it generates.

\begin{defi}\label{def:abstract-heat-kernel}
Let $r \in [1,\infty]$ and $A$ be a closed  linear operator on $L^r(X)$. A function $p=p_\cdot(\cdot,\cdot):(0,\infty)\times X\times X\to \CC$ such that
\footnote{ To avoid possible confusion, here and in the following we write $A_x$ (resp.\ $A_y$) to signify the action of the operator $A$ on $p_t$ with respect to its first coordinate (``$x$'') (resp.\ its second coordinate (``$y$'')); likewise, we occasionally write  $p_t\in D(A_x)$ (resp., $p_t\in D(A_y)$) to spell out the condition that $p_t$ belongs to the Banach space $D(A)$ and, hence, that  $A$ can be applied to $p_t$ as a function of its first (resp., second) coordinate.
}
\begin{enumerate}[(a)]
\item\label{item:measurability-heat-kernel} $p_t(\cdot,\cdot):X \times X \rightarrow \mathbb{C}$ is measurable and $p_t(x,\cdot)f(\cdot)\in L^1(X)$ for all $t>0$, all $f\in L^r(X)$, and a.e.\ $x\in X$,
\item\label{item:semigroup-maps-to-l^p} the map $x \mapsto \int_X p_t(x,y)f(y) \mathrm{d} y$ belongs to $L^r(X)$,
\item\label{item:non-differentiability-of-semigroup} $t\mapsto p_t(\cdot,y)\in C^1\left((0,\infty);L^r(X)\right)\cap C\left((0,\infty);D(A_x)\right)$ for a.e.\ $y\in X$,
\item\label{item:solution-cauchy-problem}  \(\frac{\partial }{\partial t}p_t(\cdot,y)=A_x p_t(\cdot,y)\) for all $t>0$ and a.e.\ $y\in X$,
\item\label{prop:semig} $p_{t+s}(x,y ) = \int_X p_t (x,z) p_s(z,y) \dx[\mu](z)$ for all $s,t>0$ and a.e.\ $(x,y)\in X\times X$, and
\end{enumerate} 
is called \emph{heat kernel} associated with $A$.
\end{defi}

Here we canonically regard the domain $D(A)$ of the closed, linear operator $A$ as a Banach space in its own right,  with respect to the corresponding graph norm  $\Vert \cdot \Vert_{A}$. 

The simple example of the one-dimensional Gaussian kernel associated with the (not densely defined!) operator
\[
Af:=f'',\qquad f\in W^{2,\infty}(\RR),
\]
on $L^\infty(\RR)$ shows that \autoref{def:abstract-heat-kernel} does actually generalize the traditional notion; further, much subtler examples can be found in~\cite{ElsRob08}.
Indeed, the following holds.

\begin{lemma}\label{lem:easy-semigr}
Let $r \in [1,\infty]$ and $A$ be a closed linear operator on $L^r(X)$ that is associated with a heat kernel as in~\autoref{def:abstract-heat-kernel}. Then the following assertions hold.
\begin{enumerate}[(i)]
\item\label{item:definition-integral-semigroup}
The operator family
\[
T(t)f:=\begin{cases}
\int_X p_t(\cdot,y)f(y) \dx[\mu(y)],\qquad & t>0,\\
f& t=0,
\end{cases}
\qquad f\in L^r(X),\
\]
is a semigroup of bounded linear operators on $L^r(X)$.
\item This semigroup is strongly continuous if and only if, additionally,
\begin{enumerate}[(f)]
\item\label{item:strong-continuity} $\lim\limits_{t\to 0^+}\int_X p_t(\cdot,y)f(y)\dx[\mu(y)]=f(\cdot)$ (in $L^r(X)$) for all $f\in L^r(X)$;
\end{enumerate}
in this case, $A$ is its generator, i.e., $(T(t))_{t\ge 0}=(\e^{tA})_{t\ge 0}$.
\item Let $A$ be densely defined. Then this semigroup is differentiable if and only if, additionally,
\begin{enumerate}[(c')]
\item\label{item:differentiability-of-semigroup} $t\mapsto p_t(\cdot,y)\in C^1\left((0,\infty);D(A_x)\right)$ for a.e.\ $y\in X$.
\end{enumerate}
\item 
If $r<\infty$, then its adjoint $A'$ is also associated with a heat kernel $p' = p'_\cdot(\cdot, \cdot)$ that satisfies
 \begin{equation}\label{eq:pp'}
 p'_t(x,y)=p_t(y,x) \qquad\hbox{for all }t>0\hbox{ and a.e. }x,y\in X.
 \end{equation}
\end{enumerate}
\end{lemma}

\begin{proof}
(i) The assertion follows from the properties in~\autoref{def:abstract-heat-kernel}.\eqref{item:measurability-heat-kernel}-\eqref{item:semigroup-maps-to-l^p}-\eqref{prop:semig}.

(ii) The assertion about strong continuity is obvious, since by \eqref{item:definition-integral-semigroup}
\begin{equation}\label{eq:integral-semi}
T(t)f=\int_X p_t(\cdot,y)f(y)\dx[\mu(y)]\qquad \text{for $t>0$}.
\end{equation}
It follows from the properties in~\autoref{def:abstract-heat-kernel}.\eqref{item:non-differentiability-of-semigroup}-\eqref{item:solution-cauchy-problem} that its generator is $A$.

(iii) By \cite[Chapter II, Section 4.b]{EngNag00}, $(\e^{tA})_{t\ge 0}$ is differentiable if and only if
$\e^{tA}$ is, for all $t>0$, a bounded linear operator from $L^r(X)$ to $D(A)$ such that for all $f\in L^r(X)$ the map $t\mapsto \e^{tA}f$ is continuously differentiable with
\[
\frac{\dx[\e^{tA}f] }{\dx[t]}=A\e^{tA}f,\qquad t>0.
\]
Now the claim follows from~\eqref{eq:integral-semi} by standard localization arguments.

(iv) This can be checked by a direct computation based on Fubini's Theorem.
\end{proof}

Combining~\eqref{eq:pp'} and \autoref{def:abstract-heat-kernel}.\eqref{item:non-differentiability-of-semigroup} we obtain, in particular, that
 \begin{equation}\label{eq:heat-pp'}
t\mapsto p_t(x,\cdot)\in C((0,\infty);L^{r'}(X))\qquad\hbox{ for a.e. }x\in X.
 \end{equation}


 \section{Main Results}\label{sec:main-results}
 \begin{theo}[Hölder continuity in one coordinate]\label{lem:lipschitzonecoordinate}
 Let $r \in [1,\infty)$, and let a closed, densely defined, linear operator $A$ on $L^r(X)$
 generate a differentiable semigroup $(\e^{tA})_{t \geq 0}$ on $L^r(X)$. 
 
If $D(A^k)$ is continuously embedded into $C^{0,\alpha}(X)$ for some $k \in \mathbb{N}$ and some $\alpha \in (0,1]$, then $A$ is associated with a heat kernel $p=p_\cdot(\cdot,\cdot)$ and the following assertions hold.
 \begin{itemize}
 \item[(i)] The map $X \ni x \mapsto  p_t(x,\cdot) \in D(A'_y)$
is $\alpha$-Hölder continuous
for all $t > 0$.
 \item[(ii)] The map $X\ni x \mapsto { A_y'}p_t(x,\cdot)\in L^\infty(X)$ is $\alpha$-Hölder continuous for all $t>0$.
\end{itemize}
 \end{theo}

In particular, ${ A_y'}p_t(\cdot,y): X\to \mathbb{K}$ is $\alpha$-Hölder continuous in the first coordinate for a.e.\ $y\in X$.

 \begin{proof}
In the following we will tacitly use the following: Because the semigroup $(\e^{tA})_{t \geq 0}$ is differentiable,  $\e^{tA}$ is a bounded operator mapping $L^r(X)$ into $(D(A^n), \Vert \cdot \Vert_{A^n})$ for every $n \in \mathbb{N}$ and every $t>0$ (see \cite[Chapter II, Section 4.b]{EngNag00}).
  
We can without loss of generality assume throughout the  semigroup $\e^{tA}$ to be uniformly bounded, as a scalar rescaling of the generators does not affect the smoothness of the heat kernel: this will help us to avoid technicalities when dealing with graph norms.

To begin with, observe that, for any $t>0$ and any $f\in L^r(X)$, $\e^{tA}f$ is measurable, as it belongs to $L^r(X)$; and bounded, as it lies in $D(A^k)\hookrightarrow C^{0,\alpha}(X)$. Accordingly, $\e^{tA}$ is a bounded operator mapping $L^r(X)$ to $L^\infty(X)$, i.e., it has a kernel such that by \autoref{lem:bukhv} the map $x \mapsto p_t(x,\cdot)$ is of class $L^\infty(X;L^{r'}(X))$. This means in particular that in view of \eqref{eq:heat-pp'} $p_t(x,\cdot)p_s(\cdot,y) \in L^1(X)$ for all $t,s > 0$ and a.e.\ $(x,y) \in X \times X$ and in particular property \autoref{def:abstract-heat-kernel}.\eqref{prop:semig} holds.

(i) Our proof is  inspired by that of \cite[Theorem~1]{Ouh98}.
Let $t>0$ and $f\in L^r(X)$.  Since $D(A^k)\hookrightarrow C^{0,\alpha}(X)$ continuously, we find $C_1,C_2 \geq 0$ such that
      \begin{align}
      \begin{aligned}\label{eq:abschaetzung-ah-times-semigroup}
 \| \e^{tA}f \|_{C^{0,\alpha}(X)} + \| A\e^{tA}f \|_{C^{0,\alpha}(X)} &\leq C_1 \Vert \e^{tA}f \Vert_{A^k} + C_2\Vert A\e^{tA}f \Vert_{A^k}.
          \end{aligned}
	      \end{align}
In other words, there exists  $C = C(t) \geq 0$ (depending now on $t>0$) such that 
 \begin{equation}\label{eq:Ct-diff}
\| \e^{tA}f\| _{C^{0,\alpha}(X)} + \| A \e^{tA}f\| _{C^{0,\alpha}(X)} \leq C(t) \Vert f \Vert_{L^r(X)}.
 \end{equation}
Thus, we have proved for $h \in \{0,1\}$ that
      \begin{align}\label{agn:semigrouppointwiselipschitz}
\vert  A^h \e^{tA}f(x) - A^h \e^{tA}f(x') \vert \leq C(t) \dX(x,x')^\alpha \Vert f \Vert_{L^r(X)}\qquad\hbox{for all } x,x'\in X.
      \end{align}
In view of \eqref{eq:pp'},
for $h \in \{0,1 \}$, we observe that for all $f \in D(A)$
\begin{equation}\label{eq:dub}
\begin{split}
\Big\vert \big\langle f, (A_y^h)' p_t(x,\cdot) - (A_y^h)' p_t(x',\cdot) \big\rangle_{L^r(X), L^{r'}(X)} \Big\vert
     &= \left\vert \int_X f(y) (A_y^h)'\big(p_t(x,y)- p_t(x',y)\big)\dx[\mu(y)] \right\vert \\
&= \left\vert \int_X A^hf(y) \big(p_t(x,y)- p_t(x',y)\big)\dx[\mu(y)] \right\vert \\&= \left\vert \e^{tA} A^hf(x)- \e^{tA}A^hf(x') \right\vert \\&= \left\vert A^h\mathrm{e}^{tA}f(x) - A^h\mathrm{e}^{tA}f(x') \right\vert\\& \leq C(t)  \dX(x,x')^\alpha \Vert f \Vert_{L^r(X)}.
\end{split}
\end{equation}
As 
$A$ is densely defined,
\eqref{eq:dub} also holds for every $f \in L^r(X)$ and we finally conclude that
\begin{align}\label{agn:heatkernellipschitzinxvariable}
\begin{aligned}
\big\Vert (A^h_y)' p_t(x,\cdot) - (A^h_y)' p_t(x',\cdot) \big\Vert_{L^{r'}(X)} &= \sup_{\Vert f \Vert_{L^r} =1}\big\vert \langle f, (A^h_y)' p_t(x,\cdot) - (A^h_y)' p_t(x',\cdot) \rangle_{L^r, L^{r'}} \big\vert \\&\leq C(t) \dX(x,x')^\alpha,
\end{aligned}
\end{align}
for $h \in \{0,1\}$. This shows (i).

(ii) Again because $p_t(x,\cdot)\in D(A_y')\hookrightarrow L^{r'}(X)$ for every $t>0$ and $x\in X$ 
we find by the property in \autoref{def:abstract-heat-kernel}.\eqref{item:solution-cauchy-problem}
 that
 \begin{align}\label{eq:tt2-selfadj}
 \begin{aligned}
 A_y' p_t(x,\cdot ) &= A_y' \int_X p_{\frac{t}{2}} (x,z) p_\frac{t}{2}(z,\cdot) \dx[\mu](z)\\
 & = A_y' \e^{\frac{t}{2} A_y'}p_\frac{t}{2}(x,\cdot) = \e^{\frac{t}{2}A_y'}A_y' p_\frac{t}{2}(x,\cdot) \\&= \int_X p_\frac{t}{2}(z,\cdot )\left(A_y'p_\frac{t}{2}(x,\cdot )\right)\! (z)  \dx[\mu](z),
\end{aligned}
 \qquad\hbox{for all }t>0\hbox{ and  }x\in X.
 \end{align}
 Using Hölder's inequality,
 \eqref{agn:heatkernellipschitzinxvariable} implies that for all $t>0$, all $x,x'\in X$, and a.e.\ $y\in X$ there exists $C(t/2)>0$ such that
 \begin{equation}\label{eq:tt2-selfadj-prelim}
 \begin{aligned}
 \vert A_y' p_t(x,y) - A_y' p_t(x',y) \vert &\leq \int_X \vert A_y' p_\frac{t}{2}(x,z)- A_y' p_\frac{t}{2}(x',z) \vert \vert p_\frac{t}{2}(z,y) \vert \dx[\mu](z) \\&\leq \Vert p_\frac{t}{2}(\cdot,y) \Vert_{L^r(X)} \Vert A_y'p_\frac{t}{2}(x,\cdot) - A_y' p_{\frac{t}{2}}(x',\cdot) \Vert_{L^{r'}(X)} \\&\leq \Vert p_\frac{t}{2}(\cdot,y) \Vert_{L^r(X)}C(t/2)d(x,x')^\alpha.
 \end{aligned}
 \end{equation} 
 
Because $\e^{tA}$ is bounded from $L^r(X)$  to $L^\infty(X)$ for every $t>0$ (as $D(A) \hookrightarrow C^{0,\alpha}(X)$), by \autoref{lem:bukhv} we can take $\esssup$ over $y\in X$: this yields the claim.
%
\end{proof}

\begin{rem}\label{rem:generalization-a^h}
Observe that \eqref{eq:dub} even holds for \emph{every} $h \in \NN_0$ under the assumption that 
\begin{equation}\label{eq:assumh}
(A_y^h)'p_t(x,\cdot) \in L^{r'}(X)\qquad\hbox{for all }t>0\hbox{ and a.e. }x\in X,
\end{equation}
 since $D(A^h)$ is a core for $A$ for every $h \in \NN_0$, see \cite[Proposition~II.1.8]{EngNag00}, thus leading to $\alpha$-Hölder continuity of $x\mapsto (A_y^h)'p_t(x,\cdot)$ for all $t>0$.
However, as the condition \eqref{eq:assumh} seems to be difficult to check in practice, we did not include these cases in our result.
\end{rem}
\begin{rem}\label{rem:heat-kernel-buc-second-variable}
If the underlying operator $A$ is in addition self-adjoint on $L^2(X)$, we can refine the statement of \autoref{lem:lipschitzonecoordinate} as follows: if we suppose that $D(A^k)$ is continuously embedded in $C_b(X)$ (resp., $C_{ub}(X)$) -- the space of bounded, continuous (resp., uniformly continuouos) functions on $X$ -
it is possible to deduce the continuity (resp., uniform continuity) of the corresponding heat kernel $p_t$ in the second variable, that is, of the map $y \mapsto p_t(x,y) \in \mathbb{K}$ for $t>0$ and all $x \in X$: indeed, if $D(A^k) \hookrightarrow C_{b}(X)$ for some $k$, it follows that $\e^{tA}: L^2(X) \hookrightarrow C_{b}(X)$ and therefore, by self-adjointness, $\e^{tA} = \e^{\frac{t}{2} A} \e^{\frac{t}{2} A}$ is a bounded operator from $C_{b}(X)' $ to $C_{b}(X)$: in particular,  for $x \in X$ and $\delta_x \in C_{b}(X)'$ we find $p_t(x,\cdot) = \e^{tA}\delta_x \in C_{b}(X)$ for every $x \in X$; the same approach yields uniform continuity of $p_t(x,\cdot)$, if $D(A^k) \hookrightarrow C_{ub}(X)$ for some $k$.

Under the assumption that $D(A^k) \hookrightarrow C^{0,\alpha}(X)$, the same idea may also be adapted to deduce Hölder continuity of the heat kernel with respect to the second variable,
but this would not yield that the map $x \mapsto p_t(x,\cdot)$ is of class $C^{0,\alpha}(X;D(A_y'))$.
\end{rem}

\begin{rem}
The proof of \autoref{lem:lipschitzonecoordinate} can be carried out equivalently, but sligthly more abstractly, as follows: by~\eqref{eq:Ct-diff}, it follows that there is a modulus of continuity $\omega$ such  that, for both $h \in \{0,1\}$, all $f\in D(A)$ and hence $\in L^r(X)$, and all $x,x'\in X$ s.t.\ $\mathrm{d}_X(x,x')\leq \delta$,
        \begin{align}\label{agn:semigrouppointwiselipschitz-modulus}
\vert A^h \e^{tA}f(x) - A^h \e^{tA}f(x') \vert \leq C(t)\omega(\delta) \|f\|_{L^r(X)},
      \end{align}
whence 
\begin{equation}\label{eq:dub-modulus}
\begin{split}
\Big\vert \big\langle f, (A_y^h)' p_t(x,\cdot) - (A_y^h)' p_t(x',\cdot) \big\rangle_{L^r(X), L^{r'}(X)} \Big\vert
 \leq C(t)  \omega(\delta) \Vert f \Vert_{L^r(X)}
\end{split}
\end{equation}
and eventually
\begin{align}\label{agn:heatkernellipschitzinxvariable-modulus}
\begin{aligned}
\big\Vert (A^h_y)' p_t(x,\cdot) - (A^h_y)' p_t(x',\cdot) \big\Vert_{L^{r'}(X)}\leq C(t) \omega(\delta).
\end{aligned}
\end{align}
This proves (i), as we already know that $\omega (\delta)\lesssim \delta^\alpha$ by~\eqref{eq:Ct-diff}; thus suggesting that \autoref{lem:lipschitzonecoordinate} can be adapted to deduce different smoothness properties of $x\mapsto p_t(x,\cdot)$, as soon as they can be described in terms of the properties of a modulus of continuity. This is, e.g., the case if $D(A^k)$ is continuously embedded into the space
$\Din(X):=\left\{h\in C_b(X):\frac{\omega_h}{\id}\in L^1(0,1)\right\}$ of bounded Dini continuous functions. Because
\[
h\mapsto \int_0^1 \frac{\omega_h(s)}{s}\dx[s],
\]
where
\[
\omega_h(\delta) := \sup_{\mathrm{d}_X(x,y) \leq \delta} \vert h(x) - h(y) \vert, \qquad \text{ for }h \in C^{0,\mathrm{Dini}}(X)
\]
is a semi-norm on $\Din(X)$, \eqref{eq:Ct-diff} implies again \eqref{agn:semigrouppointwiselipschitz-modulus}, \eqref{eq:dub-modulus} and finally~\eqref{agn:heatkernellipschitzinxvariable-modulus}, but this time for a modulus of continuity $\omega$ that satisfies
$\int_0^1 \frac{\omega(s)}{s}\dx[s]<\infty$. In other words: if under the assumptions of~\autoref{lem:lipschitzonecoordinate} $D(A^k)$ is embedded into $\Din(X)$ for some $k\in \NN$, then again $A$ is associated with a heat kernel and $X\ni x\mapsto p_t(x,\cdot)\in D(A'_y)$ is Dini continuous for all $t>0$.
\end{rem}

While every analytic semigroup is differentiable, the converse is not true, cf.~\cite[Section~II.4]{EngNag00}, see e.g.~\cite[Theorem A.3 and Remark A.2]{CheTri89}, \cite{Ren95}, and~\cite[Theorem~2.3 and Remark 2.7]{Bat04} for concrete examples in applications to elastic systems,  
birth-type processes, or delay differential equations, respectively. Under the stronger assumption of analyticity, we can sharpen our above result and recover a more accurate description of the dependence on $t$.

\begin{cor}\label{cor:analyt-improv}
Under the Assumptions of \autoref{lem:lipschitzonecoordinate}, let additionally the semigroup generated by $A$ be bounded analytic. Then there exist constants $C_1,C_2>0$ such that
\begin{align}\label{agn:heatkernellipschitzinxvariable-analy}
\Vert p_t(x,\cdot) - p_t(x',\cdot) \Vert_{A'} \leq (C_1 + C_2t^{-(k+1)})\dX(x,x')^\alpha
\end{align}
for all $t>0$, $x,x' \in X$.
\end{cor}
\begin{proof}
The claimed estimate follows observing that, for bounded analytic semigroups, the bound \eqref{eq:Ct-diff} can be more precisely specified: in view of \cite[Proposition~2.1.1]{Lun95}, \eqref{eq:Ct-diff} can be re-formulated as
      \begin{align}
      \begin{aligned}\label{eq:abschaetzung-ah-times-semigroup-2}
 \| \e^{tA}f \|_{C^{0,\alpha}(X)} + \| A\e^{tA}f \|_{C^{0,\alpha}(X)} &\leq \left( C_1(1+t^{-k})  + C_2(t^{-1}+t^{-k-1})\right)\Vert \e^{tA}f \Vert_{L^r(X)}.
 \end{aligned}
	      \end{align}
%
Accordingly, \eqref{agn:heatkernellipschitzinxvariable} 
becomes \eqref{agn:heatkernellipschitzinxvariable-analy}.
\end{proof}


\begin{rem}\label{rem:additional-conds}
The assumptions of \autoref{cor:analyt-improv} agree with those of \cite[Theorem~1]{Ouh98}: let us compare both results.

(i) If $r=1$, a close look at the proof of \cite[Theorem~1]{Ouh98} shows that the case of $\eta=1=\alpha$ -- corresponding to \textit{Lipschitz} continuity -- is actually admissible, although not explicitly stated there. Indeed, it is not difficult to check that both \autoref{cor:analyt-improv} and \cite[Theorem~1]{Ouh98} present two -- different but similar -- sets of assumptions that each imply Hölder (or even Lipschitz) continuity of the heat kernel of semigroups on $L^1$ spaces: in particular, we can drop Ouhabaz' assumption of ultracontractivity, at the price of replacing his  strong Gagliardo--Nirenberg-type inequality by an abstract Morrey-type inequality. We leave the details to the reader.

(ii) Let now $r>1$: then not only our assumption, but also our bound does differ from that in \cite[Theorem~1]{Ouh98}. E.g., for $r=2$ \autoref{cor:analyt-improv} can
be stated as follows: Given $\alpha\in (0,1]$, if $D(A^k)\hookrightarrow C^{0,\alpha}(X)$ there holds
\begin{equation}\label{eq:ouh-beding-2a}
\big\Vert p_t(x,\cdot) - p_t(x',\cdot) \big\Vert_{A'} 
\leq Ct^{-k} \dX(x,x')^\alpha\qquad \hbox{for all }t>0,\ x,x'\in X,
\end{equation}
whereas \cite[Theorem~1]{Ouh98} reads as follows: 
Given $\alpha\in (0,1]$, there holds
\begin{equation}\label{eq:ouh-beding-2}
\| p_t(x,\cdot) - p_t(x',\cdot)\|_{L^\infty(X)} \le  Ct^{-k} \dX(x,x')^\frac{\alpha}{2}\qquad \hbox{for all }t>0,\ x,x'\in X,
\end{equation}
provided $(\e^{tA})_{t\ge 0}$ is ultracontractive and
\begin{equation}\label{eq:ouh-beding-2-b}
\sup_{x \neq x'}\frac{|f(x)-f(x')|}{\dX(x,x')^\frac{\alpha}{2}}\le C'' \|f\|^{1-\frac{k}{\beta}}_{L^1(X)} \|(-A)^\frac{\beta}{2} f\|^{\frac{k}{\beta}}_{L^1(X)}\qquad\hbox{for all }f\in D(A^\frac{\beta}{2}).
\end{equation}
In particular, it is unclear whether the proof of \cite[Theorem~1]{Ouh98} can be adapted to yield $\alpha$-Hölder continuity for any $\alpha>\frac{1}{2}$. Observe that \eqref{eq:ouh-beding-2} is indeed implied by our estimate in \eqref{eq:ouh-beding-2a} if one additionally assumes that $D(A')$ is embedded in $L^\infty(X)$.
\end{rem}

We are  now in the position to prove our main result about joint Hölder continuity of the heat kernel associated with a generator $A$ as in \autoref{lem:lipschitzonecoordinate}.
\begin{theo}[Joint Hölder continuity]\label{thm:abstractlipschitzmeasurespaces}
Let $r \in (1,\infty)$ and let a
closed, densely defined, linear operator $A$ on $L^r(X)$
 generate a differentiable semigroup $(\e^{tA})_{t \geq 0}$ on $L^r(X)$.
Let furthermore $D(A^k) \hookrightarrow C^{0, \alpha}(X)$, and also
 $D((A')^l) \hookrightarrow C^{0,\alpha}(X)$
for some $k,l \in \mathbb{N}$ and some $\alpha\in (0,1]$.

Then the following assertions hold.

(i) $p_t\in C^{0,\alpha}(X \times X)$ for all $t>0$. 

(ii) 
If $r=2$ and $A = A'$ as operators on $L^2(X)$, then
 $A_x A_y p_t:X \times X \to \mathbb{K}$
is $\alpha$-Hölder continuous for all $t>0$. 

\end{theo}

\begin{proof}
(i) Let $t > 0$. 
As $p_t(x,\cdot) \in L^{r'}(X)$
and $p_t(\cdot,y) \in L^r(X)$
for all $t > 0$, $x\in X$, the mapping $x \mapsto p_t(x,\cdot)$ and, likewise (upon replacing the role of $A_y$ by $A_x'$), $y \mapsto p_t(\cdot,y)$ belong to $C^{0,\alpha}(X;L^{r'}(X))$ and $C^{0,\alpha}(X;L^r(X))$, respectively by \autoref{lem:lipschitzonecoordinate}.
As in the proof of \autoref{lem:lipschitzonecoordinate}.(ii), we know that $\e^{tA}\big( L^r(X)\big)\subset C^{0,\alpha}(X) \hookrightarrow L^\infty(X)$ as well as $\e^{tA'}\big( L^{r'}(X)\big)\subset C^{0,\alpha}(X) \hookrightarrow L^\infty(X)$: by \cite[Theorem 1.3]{AreBuk94}, this implies that the maps $x \mapsto p_t(x,\cdot)$ and $y \mapsto p_t(\cdot,y)$
are of class $L^\infty(X;L^{r'}(X))$ and $L^\infty(X;L^r(X))$, respectively for all $t>0$. Thus,  let
\begin{equation}\label{eq:mt}
M(t):=\max \bigg(\esssup_{x \in X} \Vert p_\frac{t}{2}(x,\cdot) \Vert_{L^{r'}(X)},\esssup_{y \in X} \Vert p_\frac{t}{2}(\cdot, y) \Vert_{L^r(X)} \bigg)
\end{equation}
which is for all $t>0$ finite by assumption. Then, for all $(x,y),(x',y') \in X \times X$ we find
\begin{align}\label{eq:joint-hoelder-without-h-estimate}
\begin{aligned}
\left\vert p_t(x,y) - p_t(x',y') \right\vert &= \bigg\vert \int_X p_\frac{t}{2}(x,z) p_\frac{t}{2}(z,y) \dx[\mu(z)] - \int_X p_\frac{t}{2}(x',z) p_\frac{t}{2}(z,y') \dx[\mu(z)] \bigg\vert 
  \\
  &\leq  \big\Vert p_\frac{t}{2}(x,\cdot) \big\Vert_{L^{r'}(X)}\big\Vert p_\frac{t}{2}(\cdot,y) - p_\frac{t}{2}(\cdot,y') \big\Vert_{L^r(X)}\\
  &\qquad + \big\Vert p_\frac{t}{2}(\cdot,y') \big\Vert_{L^r(X)}\big\Vert p_\frac{t}{2}(x,\cdot) - p_\frac{t}{2}(x',\cdot) \big\Vert_{L^{r'}(X)} \\&\leq M(t) C(t/2) (\dX(x,x') +  \dX(y, y'))^\alpha \\
  &=: L_t \cdot \mathrm{d}_{X\times X}((x,y),(x',y'))^\alpha,
  \end{aligned}
\end{align}
with $C(t)$ as in~\eqref{agn:semigrouppointwiselipschitz} and $M(t)$ as in~\eqref{eq:mt}.
This 
shows joint Hölder continuity of $p_t$
with constant $L_t := M(t) C(t/2) \geq 0$ (which only depends on $t$) and with respect to the canonical metric $\mathrm{d}_{X\times X}$ on $X \times X$.

Finally, observe that by assumption $\e^{tA}$ maps $L^r(X)$ to $L^\infty(X)$; also,  $\e^{tA'}$ maps $L^{r'}(X)$ to $L^\infty(X)$, hence by duality $\e^{tA}$ maps $L^1(X)$ to $L^r(X)$. We conclude that $\e^{tA}$ maps for all $t>0$ $L^1(X)$ to $L^\infty(X)$, i.e., its kernel is bounded: because it is $\alpha$-Hölder continuous, too, the claim follows.

(ii) To begin with, we consider the operator $A\e^{tA}$, which is bounded from $L^r(X)$ to $D(A^k)\hookrightarrow C^{0,\alpha}(X)\hookrightarrow L^\infty(X)$. By \cite[Theorem~1.3]{AreBuk94}, this means that $A\e^{tA}$ is an integral operator with kernel $q_t\in L^\infty(X;L^{r'}(X))$. However, $A\e^{tA}=\e^{tA}A$, hence
\[
\int_X q_t(x,y)f(y)\dx[y]=A\e^{tA}f(x)=\e^{tA}Af(x)=\int_X p_t(x,y) Af(y)\dx[y]=\int_X A'_y p_t(x,y) f(y)\dx[y]
\]
for all $f\in D(A)$ and hence, by density, for all $f\in L^r(X)$. In other words,
 $A'_y p_t=q_t\in L^\infty(X;L^{r'}(X))$,
 whence 
 \begin{equation}\label{eq:aptlp}
 p_t(\cdot,\cdot)\in L^\infty(X;D(A'_y))
 \qquad \hbox{for all }t>0.
 \end{equation}
Using \eqref{eq:tt2-selfadj}
we observe that $(A_xA_yp_t)(\cdot,\cdot) = \int_X (A_x p_{\frac{t}{2}})(z,\cdot)(A_y  p_{\frac{t}{2}})(\cdot, z) \dx[\mu(z)]$.
Indeed, taking into account the symmetry of the heat kernel $p = p_\cdot(\cdot,\cdot)$ as $A=A'$ (cf., \eqref{eq:pp'}), for $x,y \in X$, \eqref{eq:tt2-selfadj} (writing the integral $\int_X \cdot \dx[\mu]$ by the dual pairing $\langle \cdot, \cdot \rangle_{L^2(X),L^{2}(X)}$) yields
\begin{align}\label{eq:ax-ay-kernel}
\begin{aligned}
(A_xA_yp_t)(x,y) &= A_x \int_X p_{\frac{t}{2}} (z,y) (A_y'p_{\frac{t}{2}})(x,z) \dx[\mu](z) \\&= A_x \big\langle p_\frac{t}{2}(y,\cdot), (A_yp_{\frac{t}{2}})(x,\cdot) \big\rangle_{L^2(X),L^{2}(X)} \\&= A_x \big\langle A_y p_\frac{t}{2}(y,\cdot), p_{\frac{t}{2}}(x,\cdot) \big\rangle_{L^2(X),L^{2}(X)} \\&=A_x \e^{\frac{t}{2}A_x}A_y p_\frac{t}{2}(y,x) = \e^{\frac{t}{2}A_x} A_xA_y p_\frac{t}{2}(y,x) \\&= \big\langle p_{\frac{t}{2}}(x,\cdot), ( A_xA_y p_\frac{t}{2})(y,\cdot) \big\rangle_{L^2(X),L^{2}(X)} \\&= \big\langle (A_x  p_{\frac{t}{2}})(\cdot,x), ( A_y p_\frac{t}{2})(y,\cdot) \big\rangle_{L^2(X),L^{2}(X)}.
\end{aligned}
\end{align}
It now follows from~\autoref{lem:lipschitzonecoordinate}.(ii) that for all $(x,y),(x',y')\in X\times X$
\begin{align*}
\begin{split}
    &\vert A_xA _yp_t(x,y) - A_x A_y p_t(x',y') \vert \\
    & \qquad = \Big\vert 
    \big\langle (A_x  p_{\frac{t}{2}})(\cdot,x), ( A_y p_\frac{t}{2})(y,\cdot) \big\rangle_{L^2(X),L^{2'}(X)}   \big.\\
&\qquad\qquad -\big.
    \big\langle (A_x  p_{\frac{t}{2}})(\cdot,x'), ( A_y p_\frac{t}{2})(y',\cdot) \big\rangle_{L^2(X),L^{2'}(X)} \Big\vert \\
    & \qquad \leq \Vert A_y p_{\frac{t}{2}}(y,\cdot) \Vert_{L^2(X)} \Vert A _x p_{\frac{t}{2}}(\cdot,x) - A _x p_{\frac{t}{2}}(\cdot ,x') \Vert_{L^2(X)} \\
    & \qquad\qquad + \Vert A _x p_{\frac{t}{2}}(\cdot ,x') \Vert_{L^2(X)} \Vert A_y p_{\frac{t}{2}}(\cdot,y) - A_y p_{\frac{t}{2}}(y',\cdot ) \Vert_{L^2(X)}\\
    &\qquad \leq \widetilde{M}(t) C(t/2) (\dX(x',x ) +  \dX(y', y ))^\alpha \\
  &\qquad =: \widetilde{L}_t \cdot \mathrm{d}_{X\times X}((x',y'),(x ,y ))^\alpha,
    \end{split}
\end{align*}
where
\begin{equation}\label{eq:wt-mt}
\widetilde{M}(t):=\max \bigg(\esssup_{x \in X} \Vert A_y p_\frac{t}{2}(x,\cdot) \Vert_{L^{2}(X)},\esssup_{y \in X} \Vert A_x p_\frac{t}{2}(\cdot, y) \Vert_{L^2(X)} \bigg).
\end{equation}
This yields the claim by the same argument used in (i).
 \end{proof}

We should emphasize that the case $r=1$ is not covered by \autoref{thm:abstractlipschitzmeasurespaces} as the proof is based on applying \autoref{lem:lipschitzonecoordinate} (also) to
$A'$ to deduce that the map $y \mapsto p_t(\cdot,y)$ is of class $L^\infty(X;L^r(X))$: but -- apart from the trivial case of bounded operators -- by~\cite[Chapter II, Section 2.6]{EngNag00} the adjoint of a generator $A$ on $L^1(X)$ is not densely defined, as the semigroup generated by $A'$ is known not to be strongly continuous.

For semigroups associated with closed sesquilinear forms, 
we obtain the following; we refer to~\cite[Section~VI.3]{DauLio88} for the terminology.

\begin{cor}\label{cor:forms}
Let $V$ be a Hilbert space that is densely and continuously embedded into $L^2(X)$ and also into $C^{0,\alpha}(X)$. Let $\mathfrak a$ be a sesquilinear form  with form domain $V$ that is bounded and coercive with respect to $L^2(X)$. Then the following assertions hold.
\begin{enumerate}[(i)]
\item The corresponding operator $A$ is associated with a heat kernel $p=p_\cdot(\cdot,\cdot)$ and $p_t\in C^{0,\alpha}(X\times X)$ for all $t>0$.
\item 
 If, additionally, $\mathfrak{a}$ is symmetric, then also $A_x A_y p_t:X\times X\to \KK$ is $\alpha$-Hölder continuous.
\end{enumerate}
\end{cor}

\begin{proof}
It is well-known that the semigroup associated with $\mathfrak a$ is necessarily analytic, hence differentiable; also, the domains of both $D(A)$ and $D(A')$ are continuously embedded into $V$. Now the claims follow immediately from \autoref{thm:abstractlipschitzmeasurespaces}.
\end{proof}

  As intended in the introduction, Lipschitz continuity of the heat kernel in the setup of general one dimensional (second order) elliptic operators of the form
  \[
  A=\frac{\partial}{\partial x}\bigg(a\frac{\partial}{\partial x} + b \bigg) + c\frac{\partial}{\partial x} + d
  \]
  with coefficients $a,b,c,d\in L^\infty(\RR;\CC)$
  was -- to the best of our knowledge -- first discussed in \cite[Theorem~2.36]{AusMcITch98};
indeed, this can also be recovered by using \autoref{thm:abstractlipschitzmeasurespaces} in combination with a few auxiliary results from~\cite[Section~2]{AusMcITch98}.
In \autoref{sec:applications} we are going to exemplify our abstract methods by applying them in partially less common contexts. Before doing so, let us present one special and one generalization of our above results for two classes of problems: heat flows associated with Cheeger energies and non-autonomous evolution equations.

\subsection{Heat flows in the non-smooth setting}\label{sec:nonsmooth}
Let us briefly discuss the implications of our theory in the non-smooth setting; we refer to~\cite{Che99,Gig15} for the relevant notions. In particular, let us recall that the \textit{Cheeger energy} is defined by
\[
{\mathrm{Ch}}_2(f):=\begin{cases}
\frac{1}{2}\int_X |Df|^2_w \dx[\mu],\qquad &\hbox{if }f\in W^{1.2}(X,\dX,\mu),\\
+\infty, &\hbox{else},
\end{cases}
\]
where $|Df|^2_w$ is the 2-minimal upper gradient of $f$ and
 $W^{1,2}(X,\dX,\mu):=S^{1,2}(X,\dX,\mu)\cap L^2(X)$
 is the Sobolev space, see \cite[Chapter~2]{Gig15}. The space $(X,\dX,\mu)$ is said to have the \textit{Sobolev-to-Lipschitz property} if each $f\in W^{1,2}(X,\dX,\mu)$ with $|Df|_w\le 1$ is also Lipschitz continuous with $|f|_{C^{0,1}}\le 1$, cf.~\cite{Gig13}.

\begin{cor}\label{cor:non-smooth}
Let $(X,\dX,\mu)$ be a complete and separable infinitesimally Hilbertian space. Furthermore, let $(X,\dX,\mu)$ have the Sobolev-to-Lipschitz property, and let $W^{1,2}(X,\dX,\mu)$ be continuously embedded into $L^r(X)$
for some $r>2$.

Then the heat flow driven by $\mathrm{Ch}_2$ is associated with a jointly Lipschitz continuous heat kernel.
\end{cor}

Observe that we are not assuming $W^{1,2}(X,\dX,\mu)$ to consist of bounded functions: hence, technically speaking \autoref{cor:non-smooth} is not a direct consequence of~\autoref{cor:forms}.

\begin{proof}
By~\cite[Proposition~4.22]{Gig15}, in the infinitesimally Hilbertian case $W^{1,2}(X,\dX,\mu)$ is a Hilbert space and the Cheeger energy is a symmetric Dirichlet form with respect to $L^2(X,\mu)$; accordingly, the embedding $W^{1,2}(X,\dX,\mu)\hookrightarrow L^r(X,\mu)$ implies that the heat flow is ultracontractive, and in particular that it is associated with an $L^\infty$-heat kernel.  Now, the claim follows from~\autoref{thm:abstractlipschitzmeasurespaces}, since under our assumptions the domain of some power of the flow's (self-adjoint) generator is continuously embedded into $W^{1,2}(X,\dX,\mu)\cap L^\infty(X,\mu)\hookrightarrow C^{0,1}(X,\dX)$.
\end{proof}
   
In particular, infinitesimally Hilbertian CD$(0,N)$ spaces are known to have the Sobolev-to-Lipschitz property, see \cite[Theorem~4.10]{Gig13}. (Further sufficient properties on $(X,\dX,\mu)$ are known that imply the Sobolev-to-Lipschitz property, cf.~\cite{GigHan18,CreSou20,DelSuz22}.) On the other hand, Sobolev inequalities of the type required by \autoref{cor:non-smooth} are known to hold for fairly general classes of metric measure spaces, see~\cite[Theorem~1.3]{BjoKal22}, and also for RCD$^*(K,N)$-spaces (if $K>0$ and $N\in (2,\infty)$) and for essentially non-branching CD$^*(K,N)$-spaces with finite diameter (if $K\in \RR$ and $N\in (1,\infty)$) by~\cite{Pro15,CavMon17}.
   
\subsection{Non-autonomous problems}\label{sec:nonauto}
Let $V$ be a Hilbert space that is densely and continuously embedded in $L^2(X)$.  We consider a family $(\mathfrak{a}(t))_{t\ge 0}$ of mappings $\mathfrak{a}(t) = \mathfrak{a}(t;\cdot,\cdot)$ such that 
\begin{align}
&\label{sesquil}\hbox{$\mathfrak{a}(t;\cdot,\cdot): V\times V \to \CC$ is for all $t\ge 0$ a sesquilinear form,}\\
&\label{measurability} [0,\infty)\ni t\mapsto \mathfrak{a}(t;u,v)\in \mathbb C \hbox{ is measurable for all }u,v\in V;
\end{align}
and furthermore such that there exist constants $M, \eta> 0$ and $\omega\geq 0$ such that 
\begin{align}
\label{boundedness}|\mathfrak{a}(t;u,v)|\leq M\|u\|_V\|v\|_V\quad &\hbox{for a.e }t>0\hbox{ and }u,v\in V,\\ 
\label{ellipticity}\real \mathfrak{a}(t;u,u)+\omega\|u\|^2_{L^2(X)}\geq\eta\|u\|_V^2\quad &\hbox{for a.e }t>0\hbox{ and }u\in V,
\end{align}
hold: accordingly, by Lax--Milgram the form $\mathfrak{a}(t)$ is associated with an operator $\mathcal A(t)$ on $V'$ for a.e.\ $t>0$; observe that $D(\mathcal A(t))\equiv V$. Set 
\[
\Delta:=\{(t,s)\in (0,\infty)\times (0,\infty):s<t\}:
\]
 a function $p=p_{\cdot,\cdot}(\cdot,\cdot):\Delta\times X\times X\to \CC$ such that
\begin{enumerate}[(a)]
\item
 $p_{t,s}(\cdot,\cdot):X \times X \rightarrow \mathbb{C}$ is measurable and $p_{t,s}(x,\cdot)f(\cdot)\in L^1(X)$ for all $(t,s)\in \Delta$, all $f\in L^2(X)$, and a.e.\ $x\in X$,
 \item 
 the map $x\mapsto \int_X p_{t,s}(x,y)\dx[\mu](y)$ belongs to $L^2(X)$,
\item $t\mapsto p_{t,s}(\cdot,y)\in C^1\left((s,\infty);V'\right)\cap C\left((0,\infty);V\right)$ for all $s\in (0,\infty)$ and a.e.\ $y\in X$,
\item  $\frac{\partial }{\partial t}{p_{t,s}}(\cdot,y)=\mathcal A(t)_x p_{t,s}(\cdot,y)\) for all $(t,s)\in \Delta$ and a.e.\ $y\in X$,
\item $p_{t,s}(x,y ) = \int_X p_{t,r} (x,z) p_{r,s}(z,y) \dx[\mu](z)$ for all $t\ge r\ge s\ge 0$ and a.e.\ $x,y\in X$, and
\end{enumerate} 
is called \emph{heat kernel} associated with the family $(a (t))_{t\ge 0}$. 

If, additionally,
\begin{enumerate}[(f)]
\item $\lim\limits_{t\to s^+}\int_X p_{t,s}(\cdot,y)f(y)\dx[\mu(y)]=f(\cdot)$ (in $L^2(X)$) for all $f\in L^2(X)$,
\end{enumerate} 
then such a heat kernel induces a \textit{strongly continuous evolution family} $(U(t,s))_{(t,s)\in \Delta}$ (in the sense of~\cite[Chapter~7]{Tan79}) by
 \[
U(t,s)f(x)=\int_X p_{t,s}(x,y)f(y)\dx[\mu(y)]\qquad \hbox{for all }(t,s)\in \Delta,\hbox{ a.e. }x\in X.
 \]
We refer to~\cite{Laa18,LaaMug20} and references therein for an overview of properties of evolution families and their heat kernels, respectively. While the regularity theory of evolution families is much subtler than that of strongly continuous semigroups, see for instance \cite{AreDieLaa14}, the following is known: if there exist $\gamma\in [0,1)$ and a continuous function $\omega:[0,\infty)\to [0,\infty)$ such that
\begin{align}
&\sup_{t\ge 0} \frac{\omega(t)}{t^\frac{\gamma}{2}}<\infty\quad\hbox{and}\quad \int_0^\infty \frac{\omega(t)}{t^{1+\frac{\gamma}{2}}}\dx[t]<\infty\qquad \hbox{for all }t\ge 0,\label{eq:laas1}\\
&|\mathfrak{a}(t;f,g)-{\mathfrak a}(s;f,g)|\le \omega(|t-s|)\|f\|_V \|g\|_{V_\gamma}\qquad\hbox{for all }(t,s)\in \Delta\hbox{ and all }f,g\in V,\label{eq:laas2}\\
&\hbox{the part $A(t)$ of $\mathcal A(t)$ in $L^2(X)$ satisfies $D\left((\omega-A(t))^\frac12\right)= V$ for all $t\ge 0$,}\label{eq:laas3}
\end{align}
then $U(t,s)$ is for all $(t,s)\in \Delta$ a bounded operator from $L^2(X)$ to $V$ by~\cite[Lemma~3.3]{Laa18}.
(Here $V_\gamma$ is a complex interpolation space.) 

\begin{theo}\label{theo-hölder-nona}
Let $(\mathfrak{a}(t))_{t\ge 0}$ be a family of mappings such that \eqref{sesquil}--\eqref{measurability}--\eqref{boundedness}--\eqref{ellipticity} holds. Let additionally  \eqref{eq:laas1}--\eqref{eq:laas2}--\eqref{eq:laas3} be satisfied. 
 Let also $V$ be continuously embedded into $C^{0,\alpha}(X)$ for some $\alpha \in (0,1]$. 

Then the associated evolution family $(U(t,s))_{(t,s)\in \Delta}$ has a heat kernel $p=p_{\cdot,\cdot}(\cdot,\cdot)$. Moreover,  $X \ni x \mapsto  p_{t,s}(x,\cdot) \in V$ is $\alpha$-Hölder continuous and $p_{t,s}\in C^{0,\alpha}(X\times X)$, for all $(t,s) \in \Delta$.
\end{theo}

\begin{proof}
We will adapt the proof of \autoref{lem:lipschitzonecoordinate}.(i): By \cite[Lemma~3.1]{Laa18}, under our assumptions $U(t,s)$ is for all $(t,s)\in \Delta$ a bounded linear operator from $L^2(X)$ to $V$, hence
to $C^{0,\alpha}(X)$: in particular, $U(t,s)$ boundedly maps $L^2(X)$ to $L^\infty(X)$ and, hence, it is an integral operator with $L^\infty$-kernel. Hence, there exists  $C = C(t,s) \geq 0$ such that 
 \begin{equation}\label{eq:Ct-diff-nona}
\| U(t,s)f\| _{C^{0,\alpha}(X)} + \| A(t)U(t,s) f\| _{C^{0,\alpha}(X)} \leq C(t,s) \Vert f \Vert_{L^2(X)},
 \end{equation}
whence
      \begin{align}\label{agn:semigrouppointwiselipschitz-nona}
\vert  (-A(t))^\frac{1}{2} U(t,s)f(x) - (-A(t))^\frac{1}{2} U(t,s)f(x') \vert \leq C(t,s) \dX(x,x')^\alpha \Vert f \Vert_{L^2(X)}\qquad\hbox{for all } x,x'\in X.
      \end{align}
We can deduce as in \eqref{eq:dub} that
\begin{equation}\label{eq:dub-nona}
\begin{split}
\Big\vert \big\langle f, (-A(t)')^\frac{1}{2} p_{t,s}(x,\cdot) - (-A(t)')^\frac{1}{2} p_{t,s}(x',\cdot) \big\rangle_{L^2(X)} \Big\vert
\leq C(t,s)  \dX(x,x')^\alpha \Vert f \Vert_{L^2(X)},
\end{split}
\end{equation}
and by density of $V$ in $L^2(X)$  we finally obtain
\begin{align}\label{agn:heatkernellipschitzinxvariable-nona}
\begin{aligned}
\big\Vert p_{t,s}(x,\cdot) - p_{t,s}(x',\cdot) \big\Vert_{V}\leq C(t,s) \dX(x,x')^\alpha.
\end{aligned}
\end{align}
(Observe that \eqref{eq:laas3} implies $D\left((-A(t)')^\frac12\right)=D\left((-A(t))^\frac12\right)= V$ for all $t\ge 0$.)

Moreover, like in the proof of \autoref{thm:abstractlipschitzmeasurespaces}.(ii), we observe for $(t,s) \in \Delta$, $s < r < t$ and $(x,y), (x',y') \in X \times X$ that
\begin{align}\label{eq:joint-hoelder-without-h-estimate-non-autonomous}
\begin{aligned}
\left\vert p_{t,s}(x,y) - p_{t,s}(x',y') \right\vert &= \bigg\vert \int_X p_{t,r}(x,z) p_{r,s}(z,y) \dx[\mu(z)] - \int_X p_{t,r}(x',z) p_{r,s}(z,y') \dx[\mu(z)] \bigg\vert 
  \\
  &\leq  \big\Vert p_{t,r}(x,\cdot) \big\Vert_{L^{2}(X)}\big\Vert p_{r,s}(\cdot,y) - p_{r,s}(\cdot,y') \big\Vert_{L^{2}(X)}\\
  &\qquad + \big\Vert p_{r,s}(\cdot,y') \big\Vert_{L^2(X)}\big\Vert p_{t,r}(x,\cdot) - p_{t,r}(x',\cdot) \big\Vert_{L^{2}(X)}.
  \end{aligned}.
\end{align}
Since the operators $U(t,r)$ and $U(r,s)$ map $L^2(X)$ continuously to $V$ and thus -- as $V \hookrightarrow C^{0,\alpha}(X)$ by assumption -- to $L^\infty(X)$, the maps $x \mapsto p_{t,r}(x,\cdot)$ and $y \mapsto p_{r,s}(\cdot,y)$ both belong to $L^\infty(X;L^2(X))$. Therefore, using that $V \hookrightarrow L^2(X)$ continuously as well as \eqref{agn:heatkernellipschitzinxvariable-nona}-\eqref{eq:joint-hoelder-without-h-estimate-non-autonomous} finally yields the joint $\alpha$-Hölder continuity of $p_{t,s}$.
\end{proof}

\section{Some applications}\label{sec:applications}
\subsection{Laplacians on bounded domains}\label{subs:laplacians-domains}

Consider a general open domain  $\Omega\subset \RR^d$ that satisfies the strong local Lipschitz condition, see \cite[Section~4.9]{AdaFou03}.
Then it is known that 
\[
H^m(\Omega)\hookrightarrow C^{0,m-\frac{d}{2}}(\overline{\Omega}) \quad\hbox{if } \:\: m-1<\frac{d}{2}<m
\]
and also
\[
H^{1+\frac{d}{2}}(\Omega)\hookrightarrow C^{0,\alpha}(\overline{\Omega})  \quad\hbox{{for all } }\alpha<1,
\]
as long as the exponents of the Sobolev spaces are integers, see \cite[Theorem 4.12, Part II]{AdaFou03}. While in none of these cases the choice $\alpha=1$ is admissible, hence this does not yield an embedding into the space of Lipschitz continuous functions, the Besov space $B^s_{2,\infty}(\Omega)$ satisfies
\[
B^s_{2,\infty}(\Omega)\hookrightarrow C^{0,s-\frac{d}{2}}(\overline{\Omega})\quad\hbox{if }\lceil s\rceil-1\le \frac{d}{2}<s\le \lceil s\rceil,
\]
by \cite[Theorem~7.37]{AdaFou03}, hence in particular
\[
B^{1+\frac{d}{2}}_{2,\infty}(\Omega)\hookrightarrow C^{0,1}(\overline{\Omega}) \qquad \text{for even $d$}.
\]
We conclude that, by \autoref{lem:lipschitzonecoordinate}, if $A$ generates a differentiable semigroup on $L^2(\Omega)$ for $d$ even, and if $D(A^k)$ is continuously embedded into a Besov space $B^s_{2,\infty}(\Omega)$ for $s$ large enough, then the heat kernel is Lipschitz continuous of exponent $s-\frac{d}{2}$ in each coordinate.

If $\Omega$ satisfies instead the cone condition, then by \cite[Theorem 4.12, Part I, Case A]{AdaFou03}
\begin{equation}\label{eq:adamsfour-sobolev2}
H^{m}(\Omega)\hookrightarrow W^{1,\infty}(\Omega)\quad\hbox{if }m>\frac{d}{2}+1
\end{equation}
again as long as $m$ is  integer; and it is also known that $W^{1,\infty}(\Omega)=C^{0,1}(\overline{\Omega})$ whenever $\overline{\Omega}$ is a \textit{quasiconvex} domain (i.e., there exists $C>0$ such that any two points $x,y\in\overline{\Omega}$ can be joined by a curve $\gamma$ whose support is contained in $\overline{\Omega}$ and whose length does not exceed $C\dist(x,y)$, see \cite[Theorem~4.1]{Hei05}).
Again, we conclude by \autoref{lem:lipschitzonecoordinate}, that if $A$ generates a differentiable semigroup on $L^2(\Omega)$, and if $D(A^k)\hookrightarrow H^m(\Omega)$ for some $m>\frac{d}{2}$, then the heat kernel is Lipschitz continuous in each coordinate.

Checking that $D(A^k)$ is continuously embedded into Besov or Sobolev spaces of order high enough is a standard application of boundary regularity theory. Let us show how to deal with a class of second order elliptic operators -- even though more refined results are actually available \cite[Section~3]{ElsOuh19} -- and even to higher order pseudodifferential operators.

\begin{prop}
Let $\Omega\subset \RR^d$ be a bounded, open, quasiconvex domain with boundary of class $C^{2k}$ for some $k \in \mathbb{N}_0$ with $k > \frac{d+2}{4}$.
 Let 
\[
A:u\mapsto \nabla\cdot(a\nabla u)+b\cdot \nabla u + cu
\] 
be a uniformly elliptic second-order differential operator
with Dirichlet boundary conditions and measurable
coefficients $a\in C^{2k-1}(\overline{\Omega};\RR^{d\times d})$, $b\in C^{2k-1}(\overline{\Omega};\RR^d)$, $c\in C^{2k-1}(\overline{\Omega};\RR)$.
Then  $\overline{\Omega} \ni x\mapsto p_t(x,\cdot)\in H^2(\Omega)$ is Lipschitz continuous.
Furthermore, $p_t\in C^{0,1}(\overline{\Omega}\times \overline{\Omega})$, and even $A_x A_y p_t\in C^{0,1}(\overline{\Omega}\times \overline{\Omega})$ if $a(x)$ is Hermitian for each $x\in\Omega$ and $b\equiv 0$.

If $b\equiv 0$ and $c\le 0$, then also the heat kernel $p_t^{(\gamma)}$ associated with the power $-(-A)^\gamma$ enjoys for all $\gamma>1$ the same regularity properties if $a$ is Hermitian and all $\gamma > 1$ or else for general $a$ but $\gamma$ small enough.
 \end{prop}
 
 \begin{proof}
By~\cite[Theorem~3.1]{AreEls97}, $A$ generates on $L^2(\Omega)$ a strongly continuous, analytic semigroup $(\e^{tA})_{t \ge 0}$ with integral kernel $p_t$ that belongs to $L^\infty(\Omega\times \Omega)$ for all $t>0$. Now let $k \in \NN$. Then, applying \cite[Theorem~5 in Section~6.3]{Eva10}  iteratively for $m=0,1,\ldots,2(k-1)$, the elliptic problem
\[
\left\{
\begin{split}
A^ku&\in L^2(\Omega),\\
u_{|\partial\Omega}&=0,
\end{split}
\right.
\]
can be seen to enjoy boundary regularity and, hence, $u\in H^{2k}(\Omega)$. In other words, $D(A^{k}) \hookrightarrow H^{2k}(\Omega)$, hence by \eqref{eq:adamsfour-sobolev2} -- for $k>\frac{d+2}{4}$
 --  $D(A^{k}) \hookrightarrow W^{1,\infty}(\Omega)\hookrightarrow C^{0,1}(\overline{\Omega})$, by quasiconvexity.
By \autoref{lem:lipschitzonecoordinate}.(i), we conclude that  
$p_t$ is Lipschitz continuous up to  the boundary with values in $D(A')=H^2(\Omega)$.

The joint Lipschitz continuity of the heat kernel follows from \autoref{thm:abstractlipschitzmeasurespaces}.(i), and in the self-adjoint case from \autoref{thm:abstractlipschitzmeasurespaces}.(ii).

If $b\equiv 0$ and $c\ge 0$, then $A$ is associated with a coercive form and, hence, $-(-A)^\gamma$ can be defined via $H^\infty$-functional calculus. In particular, for all $\gamma> 1$ it generates a strongly continuous, analytic semigroup and its domain is contained in $D(A)$, whence the assertion follows.
\end{proof}

\subsection{Schrödinger operators on metric graphs}\label{sec:metricgraphs}
Let $\mathcal{G}$ be a connected, possibly infinite but locally finite compact metric graph with 
edge set $ \mathsf{E}$, vertex set $ \mathsf{V}$, and edge lengths $(\ell_\me)_{\me\in\mE}$.  We assume in the following that 
\begin{equation}\label{eq:bdd-geom}
\inf_{\me\in\mE}\ell_\me>0.
\end{equation}
We endow $\Graph$ with the canonical metric measure structure: more precisely $(\Graph,d_\Graph,\mu_\Graph)$ is obtained letting $d_\Graph$ be the shortest path metric, and $\mu_\Graph$ be the direct sum of the Lebesgue measure on each interval.
We can thus consider the spaces $L^2(\Graph)$ and  $C(\mathcal{G})$ of functions on $\Graph$ that are square integrable and continuous with respect to the measure and metric structure of $\Graph$, respectively.  Also, we look at the Sobolev space
 \(
W^{1,2}(\mathcal{G}) := \bigoplus_{\mathsf{e} \in \mathsf{E}} W^{1,2}(0,\ell_\mathsf{e}) \cap C(\mathcal{G})
\)
and, for any subset $\mV_{\mathsf D}$ of $\mV$, at
\(W^{1,2}_0(\mathcal{G};\mVD):=\{f\in W^{1,2}(\Graph):f(\mv)=0\hbox{ for all }\mv\in\mVD\}\):
we refer to \cite{Mug19} for a more detailed introduction.

We consider the self-adjoint, positive semi-definite operator
 $\Delta^{\Graph;\mVD}:=\Delta^\Graph := \frac{\dx[]^2}{\dx[x]^2}$ on $L^2(\Graph)$ associated with the closed quadratic form
$\mathfrak{a}_\mathcal{G}(f) := \Vert f' \Vert_{L^2(\mathcal{G})}^2$
with form domain $W^{1,2}_0(\Graph;\mVD)$: this is the free Laplacian with mixed vertex conditions  (\emph{Dirichlet} at $\mVD$ and \emph{standard}, i.e., continuity/Kirchhoff, at $\mV\setminus \mVD$).

Now, $\Delta^{\Graph;\mVD}$ generates an analytic, strongly continuous semigroup $(\e^{t\Delta^\Graph})_{t \geq 0}$
 that maps $L^2(\mathcal{G})$ to the form domain $W^{1,2}_0(\mathcal{G}; \mV_{\mathsf D})$ and hence by \cite[Lemma~3.2]{KosMugNic22} to $C^{0,\frac12}(\Graph)\cap L^\infty(\Graph)$.
By duality we see that $\e^{t\Delta^\Graph}$ is, for all $t>0$, a bounded linear operator from $L^1(\mathcal{G})$ to $L^\infty(\mathcal{G})$, hence an integral operator with kernel $p_t^\mathcal{G} \in L^\infty(\mathcal{G} \times \mathcal{G})$.
 Also, by~\cite[Theorem~6]{BecGreMug21} $p^\Graph_t$ is edgewise smooth, more precisely: $p^\Graph_t\in C^\infty(\mathcal E\times \mathcal E)$, where $\mathcal E:=\bigsqcup_{\me\in \mE}(0,\ell_\me)$.
Let us now refine these observations.

\begin{prop}\label{prop:resume-qgraph}
The heat kernel $p_t^{\Graph}$ associated with $-\Delta^\Graph$ satisfies the following properties:
\begin{itemize}
\item[(i)] $x \mapsto \Delta_y^\Graph p^\Graph_t(x,\cdot)$ is  for every $t>0$ Lipschitz continuous, as a mapping from $\Graph$ to both $L^2(\Graph)$ and $L^\infty(\Graph)$;
\item[(ii)]  both $p_t^\Graph$ and $\frac{\partial^2}{\partial x^2}\frac{\partial^2}{\partial y^2} p_t^\mathcal{G}:{\Graph}\times{\Graph}\to \KK$ are Lipschitz continuous for every $t>0$;
\item[(iii)]   $x \mapsto p_t^\Graph(x,\cdot)$ is  for every $t>0$ Lipschitz continuous, as a mapping from $\Graph$ to $W^{2,2}(\mathcal E)\cap W^{1,2}(\Graph)$; and in fact
\begin{align}\label{agn:heatkernellipschitzinxvariable-analy-appl}
\Vert p_t^\Graph(x,\cdot) - p_t^\Graph(x',\cdot) \Vert_{W^{2,2}} \leq (C_1 +C_2 t^{-2})d_\Graph(x,x')
\end{align}
for some $C_1,C_2>0$, all $t>0$, and all $x,x'\in \Graph$.
\end{itemize}
\end{prop}

\begin{proof}
By \cite[Lemma 3.7]{MugPlu23},   \(
C(\mathcal{G}) \cap \bigoplus_{\me \in \mE} W^{2,2}(0,\ell_\me)
  \) 
  is continuously embedded in $C^{0,1}(\mathcal{G})$, and
in particular, $D(\Delta^\Graph) \hookrightarrow C^{0,1}(\mathcal{G})$ continuously
(this was stated in \cite{MugPlu23} for the case of finite metric graphs, but it is easy to see that the proof remains valid in the infinite case, too).
Thus, \autoref{lem:lipschitzonecoordinate}, \autoref{thm:abstractlipschitzmeasurespaces}, and \autoref{cor:analyt-improv} imply (i), (ii), and (iii), respectively.
\end{proof}

At the risk of being  pleonastic, let us stress that functions in the domain of $\Delta^\Graph$ are Lipschitz continuous, as we have just seen; but in view of the Kirchhoff conditions satisfied by their derivatives, they are \textit{not} continuously differentiable as long as $\Graph$ contains  vertices of degree higher than 2. In particular, $p_t^\Graph(\cdot,y)$ is \textit{not} 
continuously differentiable, for any $y\in \Graph$, unless $\Graph$ is a path graph or a loop.

\begin{rem}
It was shown in~\cite[Theorem~5.2]{KosMugNic22}  that $p_t^\Graph\in C^{0,\frac{1}{2}} ({\Graph}\times{\Graph})$: 
this also follows from~\autoref{thm:abstractlipschitzmeasurespaces}.(i) and the one-dimensional Morrey inequality ${W^{1,2}}(\RR)\hookrightarrow C^{0,\frac12}(\RR)$.
This smoothness property is weaker than the one obtained in \autoref{prop:resume-qgraph} for the plain Laplacian, but it immediately extends to uniformly elliptic operators whose form domain is ${W^{1,2}}(\Graph)$,
like the general operators
\[
A:=\frac{\partial}{\partial x}c\frac{\partial}{\partial x}+V
\]
 considered in \cite[Section~6.5]{Mug14}, with $c\in L^\infty(\Graph) $ and $V\in L^1(\Graph)$. In view of~\autoref{theo-hölder-nona}, the heat kernels $p_{t,s}$ associated with
 time-dependent families of operators
 \[
A(t):=\frac{\partial}{\partial x}c(t)\frac{\partial}{\partial x}+V(t),\quad t\ge 0,
\]
satisfy $p_{t,s}\in C^{0,\frac12}(\Graph\times \Graph)$ for all $(t,s) \in \Delta$ provided $[0,\infty)\ni t\mapsto c(t)\in L^\infty(\Graph)$ and $[0,\infty)\ni t\mapsto V(t)\in L^1(\Graph)$ are uniformly continuous with respect to time and their moduli of continuity satisfy~\eqref{eq:laas1}. 
\end{rem}

The same arguments allow us to discuss more general \emph{magnetic Schrödinger operators} induced by some \emph{electric potential} $q = (q_\me)_{\me \in \mE} \in L^2(\mathcal{G};\CC)$ with $\essinf\limits_{x\in \Graph} \real q(x) \ge q_0>-\infty$ and some \emph{magnetic potential} $B =(B_\me)_{\me \in \mE} \in L^\infty(\mathcal{G};\mathbb{R})$: this is by definition the operator $H^\mathcal{G} = (\frac{\mathrm{d}}{\mathrm{d}x} - \icomp B)^2 - q$ associated with the form $\mathfrak{a}_{\mathcal{G};B,q}(f) := \Vert  f' - \icomp Bf\Vert_{L^2(\mathcal{G})}^2 + \int_\Graph q|f|^2\dx[x]$ with form domain given once again by $W^{1,2}_0(\mathcal{G};\mVD)$ (see \cite[Remark~4.4.(2)]{EgiMugSee23}). 

Indeed, $D(H^\mathcal{G})\hookrightarrow W^{1,2}_0(\mathcal{G};\mVD)\hookrightarrow L^\infty(\Graph)$
and one can show that every $f \in D(H^\mathcal{G})$ is twice weakly differentiable and on each edge $\me \in \mE$ with
\[
f_\me'' = (H^\mathcal{G} f)_\me + \icomp B_\me f_\me' + \icomp B_\me(f_\me' - \icomp B_\me f_\me) + q_\me f_\me,
\]
which readily implies that $f_\me \in W^{2,2}(0,\ell_\me)$ for every $\me \in \mE$.  This setting can be further generalized by imposing so-called \emph{$\delta$-type} (instead of standard) vertex conditions with strengths $\sigma = (\sigma_\mv)_{\mv \in \mV \setminus \mV_{\mathsf D}}\in \CC^{\vert \mV\setminus\mVD \vert}$
 such that $\inf_{\mv\in\mV\setminus\mVD } \real \sigma_\mv \ge \sigma_0>-\infty$
or even more general vertex conditions, as long as the corresponding Schrödinger operator $H_\mathcal{G}$ still generates a differentiable semigroup.

These arguments fail for  magnetic potentials $B \in L^r(\mathcal{G}) := \bigoplus_{\me \in \mE} L^r(0,\ell_\me)$ with $r \in [2,\infty)$, but in these cases the form domain is also contained in $W^{1,2}(\mathcal{G}) \hookrightarrow C^{0,\frac{1}{2}}(\mathcal{G})$ (again, see \cite[Remark~4.4.(2)]{EgiMugSee23}) and we can at least deduce from \autoref{cor:forms} (joint) $\frac{1}{2}$-Hölder continuity of the corresponding heat kernel.

\subsection{Laplacians on fractals}\label{sec:fractals}
Let $N \in \mathbb{N}$ and $F=(F_i)_{i=1,\dots,N}$ a so-called \emph{iterated function system}, that is, each $F_i: \mathbb{R}^d \rightarrow \mathbb{R}^d$ is a contraction on $\mathbb R^d$. It is known 
that there exists a unique compact subset $K \subset \mathbb{R}^d$ such that 
\begin{align}\label{eq:self-similarity-condition}
K=F(K) := \bigcup_{i=1}^N F_i(K)
\end{align}
often known as the \emph{self-similar identity}, see e.g.~\cite[(1.1.8)]{Str06}. Moreover, defining
\[
V_k = \bigcup_{\vert \omega \vert =k} F_\omega(K) \qquad \text{with $F_\omega = F_{\omega_1} \circ \dots \circ F_{\omega_k}$,}
\]
where $\omega \in \{1,\dots,N\}^k$ is a word of length $\vert \omega \vert = k$ for $k \in \mathbb{N}$ and considering the relation $\sim_k$ on $V_k$ given by
\[
x \sim_k y \:\: :\Leftrightarrow \:\: \text{there is a word $\omega$ of length $\vert \omega \vert$ such that $x,y \in F_\omega(K)$}
\]
one can follow the procedures represented in \cite[Section~1.4]{Str06} resp.\ \cite[Section~4.2]{Rui13} to define an energy form $\mathcal{E}_K$ which is closed on some form domain $D(\mathcal{E}_K)$, see \cite[Theorem~1.4.2]{Str06} resp.\ \cite[Theorem~4.2.4]{Rui13}. Thus, given some (probability) measure $\mu_K$ on $K$, this yields a self-adjoint operator $-\Delta^K$ associated with $\mathcal{E}_K$ which is called the \emph{Laplacian on the cell (or fractal) $K$}. Moreover, one can define the \emph{resistance metric} $R_K: K \times K \rightarrow [0,+\infty]$ on $K$ through
\begin{align}\label{eq:resistance-metric}
R_K(x,y) := \sup_{u \in D(\mathcal{E}_K)} \frac{\vert u(x) - u(y) \vert^2}{\mathcal{E}_K(u)} \qquad \text{for $x,y \in K$,}
\end{align}
in other words $(K,R_K,\mu_K)$ is a generalized metric measure space.
The case of the Sierpiński gasket, $K={S\!G}$, is thorougly discussed in~\cite[Chapter~1]{Str06}.

It is now possible to deduce Hölder continuity of the heat kernel associated with $-\Delta^{S\!G}$.
The most common choice for a measure on the Sierpiński Gasket is given by the normalized $d$-dimensional Hausdorff measure,
i.e. the measure given by 
\begin{equation}\label{eq:measure-sierp}
\mu_K(B) := \frac{1}{\mathcal{H}^d(K)} \mathcal{H}^d\vert_K(B) \qquad \text{for every Borel set $B \subset \mathbb{R}^2$,}
\end{equation}
where $\mathcal{H}^d$ denotes the $d$-dimensional Hausdorff measure of dimension $d=\frac{\ln(2)}{\ln(3)}$. 

\begin{prop}\label{prop:resume-fractals}
The Laplace operator $-\Delta^{S\!G}$ on the Sierpiński gasket $S\!G$ is associated with a bounded heat kernel $p_t \in L^\infty(S\!G \times S\!G)$ which for all $t>0$ satisfies the following properties:
\begin{itemize}
\item[(i)] $x \mapsto \Delta^{S\!G}_y p_t(x,\cdot) \in L^\infty(S\!G)$ is for every $t > 0$ $\frac{1}{2}$-Hölder continuous with respect to the resistance metric $R_{S\!G}$;
\item[(ii)] both $p_t$ and $\Delta_x^{S\!G}\Delta_y^{S\!G} p_t: {S\!G} \times {S\!G} \rightarrow \mathbb{K}$ are $\frac{1}{2}$-Hölder continuous for every $t >0$ with respect to the resistance metric $R_{S\!G}$.
\end{itemize}
\end{prop}
\begin{proof}
According to \cite[(1.6.3)]{Str06}, it follows that $\vert u \vert_{C^{0,\frac{1}{2}}({S\!G})} \leq \mathcal{E}_{S\!G}(u)^\frac{1}{2}$ for every $u \in D(\mathcal{E}_{S\!G})$. As ${S\!G}$ is compact with respect to the resistance metric, it follows that $D(\mathcal{E}_{S\!G}) \hookrightarrow C^{0,\frac{1}{2}}({S\!G})$ and therefore $D(\Delta^{S\!G}) \hookrightarrow C^{0,\frac{1}{2}}({S\!G})$. Thus $(\e^{-t\Delta^{S\!G}})_{t \ge 0}$ generates an analytic semigroup consisting of operators mapping $L^2({S\!G})$ to $L^\infty({S\!G})$ and, hence, of integral operators yielding integral kernels $p_t \in L^\infty({S\!G} \times {S\!G})$ such that for all $t>0$ the map $x \mapsto \Delta_y^{{S\!G}}p_t(x,\cdot)$ belongs to $C^{0,\frac{1}{2}}({S\!G};L^\infty({S\!G}))$ according to \autoref{lem:lipschitzonecoordinate} and $p_t, \Delta_x^{S\!G}\Delta_y^{S\!G} p_t$ is jointly $\frac{1}{2}$-Hölder continuous by~\autoref{cor:forms}. 
\end{proof}

\begin{rem}
(i) We can generalize \autoref{prop:resume-fractals} to consider the so-called \emph{Hanoi attractors} $H\! A_\beta$, $\beta\in [0,\frac{1}{3})$: these objects have been thoroughly discussed in \cite[Chapter~4]{Rui13} (see also \cite{FreRui17} and references therein) and for $\beta=0$ include the Sierpiński Gasket as special case: see \cite[Figure~1]{FreKigRui18} for an illustration of these fractals.

In this case one has to choose an appropriate modification of the measure in~\eqref{eq:measure-sierp}, see, e.g., \cite[Section~4.3.1]{Rui13} to guarantee that, again, $D(\Delta_{H\! A_\beta}) \hookrightarrow C^{0,\frac{1}{2}}(H\! A_\beta)$ by \cite[p.~59]{Rui13} for $\beta \in (0,\frac{1}{3})$.

(ii) Note that -- using the proof of \autoref{prop:resume-fractals} -- it is also possible to deduce Hölder continuity of the corresponding heat kernel with respect to the usual Euclidean distance $\vert \cdot \vert_K$ on $K\in \{S\!G, H\! A_\beta\}$ where $K$ is embedded in the Euclidean space $\mathbb{R}^2$ but with different (and more complicated!) Hölder exponents. More precisely, it is true in the Sierpiński gasket case $K=S\!G$ that the functions in the form domain $D(\mathcal{E}_K)$ are Hölder continuous of exponent $\alpha := \ln(\frac{5}{3}) / \ln(2)$, see, \cite[p.~19]{Str06}, whereas in the Hanoi attractor case $K=H\! A_\beta$ for some $\beta \in (0, \frac{1}{3})$ it is known that functions in the form domain are Hölder continuous of exponent $\alpha_\beta := \frac{1}{2} \ln(\frac{5}{3}) / \ln(\frac{2}{1-\beta})$ with respect to the Euclidean distance $\vert \cdot \vert_K$ on $K$, see, \cite[Proposition~2.12]{FreRui17}).

(iii) The statement of \autoref{prop:resume-fractals} can also be formulated for more general sets $X$ which yield so-called \emph{resistance forms} (and therefore a canonical resistance metric as defined for $K$ in \eqref{eq:resistance-metric}), whenever the underlying energy form is closed and, hence, it is associated with a self-adjoint operator.
We refer to \cite{Kig02} and references therein for more details on resistance forms.
\end{rem}

\subsection{Structurally damped wave equations}\label{sec:damped-wave}

For a general metric measure space $X$, we consider the damped wave equation
\begin{equation}\label{eq:damped-wave}
\frac{\partial^2 u}{\partial t^2}(t,x)=-D u(t,x)-B\frac{\partial u}{\partial t}(t,x),\qquad t>0,\ x\in X.
\end{equation}

Let us now assume that there exists a densely defined, closed, and invertible operator $C$ on $L^2(X)$ such that $CC^*=D$: in particular, $D$ is self-adjoint and positive definite. We also assume $B$ to be a self-adjoint operator on $L^2(X)$, such that $\rho D^\alpha=B$ for some $\alpha\in (0,1]$ and $\rho>0$.

Then, \eqref{eq:damped-wave} can be equivalently re-written as
\[
\frac{\partial }{\partial t}\begin{pmatrix}
v\\ u
\end{pmatrix}(t,x)=\mathcal A \begin{pmatrix}
v\\ u
\end{pmatrix}(t,x),\qquad t>0,\ x\in X,
\]
where the operator matrix $\mathcal A$ on $L^2(X)\times L^2(X)$ is given by
\[
\mathcal A:=\begin{pmatrix}
0 & C^*\\ -C & -B
\end{pmatrix},\qquad D(\mathcal A)\supset D(C)\times \left(D(C^*)\cap D(B)\right).
\]
We recall that $L^2(X)\times L^2(X)\simeq L^2(X\sqcup X)$: hence, in order to apply our general theory, we resort to the metric measure space 
\[
\mathbb X:=X\sqcup X:
\]
 because $X\sqcup X$ is a disconnected metric space, we will effectively deal with a generalized metric, see \autoref{exa:disj-un} below for the notion of Hölder continuity in the disconnected generalized metric measure space $X\sqcup X$.

It was shown in \cite{CheTri89}, see also \cite[Chapter~6]{XiaLia98}, that, upon closure, $\mathcal A$ generates on $L^2(X)\times L^2(X)$ a semigroup that is analytic in the range $\alpha \in [\frac{1}{2},1]$, and differentiable (but not analytic) in the range $\alpha\in (0,\frac{1}{2})$. It is known, see~\cite{CheTri90}, that  $\mathcal A$ is closed with domain
$D(\mathcal A)=D(C)\times D(C)$ in the \textit{structurally damped} case of $B=\rho C=\rho C^*$ (and, hence, $\alpha=\frac12$).

In particular, if $p_t$ denotes the (matrix-valued!) heat kernel of $\e^{t\mathcal A}$ and 
$p^{(21)}_t,p^{(22)}_t$ denote its components on the second row,
 one can observe that
\[
u(t,x) := \int_{X} \left(p^{(21)}_t(x,y)f_1(y) + p^{(22)}_t(x,y)f_2(y)\right) \dx[y],\qquad x\in X,
\]
where $f_1(x) = v(0,x)$ and $f_2(x) = u(0,x)$ represent the initial data of \eqref{eq:damped-wave} for $x \in X$: accordingly, $u$ inherits the smoothness of $p^{(21)}_t,p^{(22)}_t$: let us illustrate this in a simple case.

\begin{prop}
Let $\Delta:=-\frac{\dx[]^2}{\dx[x]^2}$ with $D(\Delta):=H^2(0,1) \cap H^1_0(0,1)$. Then the semigroup generated by
\[
\mathcal A:=
\begin{pmatrix}
0 & \Delta\\ -\Delta & -\Delta
\end{pmatrix}
\]
on $L^2(\mathbb X)$ has a bounded, jointly Lipschitz continuous integral kernel $p_t$ for all $t>0$. Also, $\mathbb x\mapsto p_t(\mathbb x,\cdot)\in C^{0,1}\left(\mathbb X;D(\mathcal A_{\mathbb y})\right)$ for all $t>0$.
Furthermore, 
\begin{equation}\label{eq:smooth-wave}
\left\|
\begin{pmatrix}
0 & \Delta \\ -\Delta & -\Delta
\end{pmatrix}
p_t(\mathbb{x},\cdot)
-
\begin{pmatrix}
0 & \Delta \\ -\Delta & -\Delta
\end{pmatrix}
p_t(\mathbb{x}',\cdot)\right\|_{L^2(\mathbb X)}\le (C_1 t^{-1}+C_2 t^{-2})d_{\mathbb{X}}(\mathbb{x},\mathbb{x}'),
\end{equation}
for all $t>0$ and all $\mathbb{x},\mathbb{x}'\in \mathbb X$, where $d_{\mathbb{X}}(\mathbb{x},\mathbb{x}')$ is defined as in~\eqref{eq:dist-sqcup} below.
\end{prop}

\begin{proof}
By the results in~\cite{CheTri89,CheTri90} applied to $B=C=C^*:=\Delta$, $\mathcal A$ generates an analytic semigroup.
We are going to apply our general theory with $X:=(0,1)$ and hence, as in \autoref{exa:disj-un} below, $\mathbb X:=X\sqcup X=(0,1)\sqcup (0,1)$. 

By analyticity, each $\e^{t\mathcal{A}}$ maps $L^2(\mathbb X) \simeq L^2(0,1) \times L^2(0,1)$ into $D(\mathcal A) \subset H^2(0,1) \times H^2(0,1)\hookrightarrow L^\infty({\mathbb X})$.	
Now, observe that the operator matrices $\mathcal A,\mathcal A'$ are similar: indeed
\[
\mathcal A'=
\begin{pmatrix}
0 & -\Delta\\ \Delta & -\Delta
\end{pmatrix}
=\begin{pmatrix}
\Id & 0\\ 0 & -\Id
\end{pmatrix}
\begin{pmatrix}
0 & \Delta\\ -\Delta & -\Delta
\end{pmatrix}
\begin{pmatrix}
\Id & 0\\ 0 & -\Id
\end{pmatrix}=:U\mathcal AU^{-1}
\]
Accordingly, $\mathcal A'$ generates an analytic semigroup, too, that likewise maps $L^2(\mathbb X)$ to $L^\infty(\mathbb X)$, and by duality $\e^{t\mathcal A}$ maps $L^1(\mathbb X)$ to $L^\infty(\mathbb X)$.
This implies that the integral kernel $p_t$ of $\e^{t\mathcal A}$ is indeed of class $L^\infty(\mathbb{X} \times \mathbb{X})$, for all $t>0$.
 Moreover, the corresponding integral kernel $p_t$ is jointly Lipschitz continuous by~\autoref{thm:abstractlipschitzmeasurespaces}.(i)
because $D(\mathcal A)\subset H^2(0,1) \times H^2(0,1)\hookrightarrow C^{0,1}([0,1])\times C^{0,1}([0,1]) \simeq C^{0,1}({\mathbb X})$.

It can be proved by induction that $D(\mathcal A^k)=H^{2k}(0,1)\times H^{2k}(0,1)$ as well as $D\left((\mathcal A')^k\right)=H^{2k}(0,1)\times H^{2k}(0,1)$, for all $k\in \NN_0$: this yields the second assertion. Finally, \eqref{eq:smooth-wave} follows applying \autoref{cor:analyt-improv}.
\end{proof}
Extending this result to more general elliptic operators $B,C$ and higher dimensional domains $\Omega\subset \RR^d$ is only a matter of formulating the appropriate smoothness conditions on the coefficients of $B,C$ and the domain $X$.

\begin{rem}\label{exa:disj-un}
We stress that
there is no change in the notion of Hölder continuity upon allowing for \emph{generalized} metric measure spaces, i.e., for measure spaces endowed with a metric that can attain the value $+\infty$: this setup has been crucial in the analysis of the wave equation \eqref{eq:damped-wave}.

Let $X$ be a metric measure space and endow the disjoint union $\mathbb{X}:=X\sqcup X$ with respect to the canonical product measure, see~\cite[Section I.6]{Con85}, and the (generalized) metric  $d_{\mathbb X}$ on $\mathbb{X} := X \sqcup X$ defined via
\begin{equation}\label{eq:dist-sqcup}
d_{\mathbb X}(\mathbb{x},\mathbb{y}) := \begin{cases} d(\mathbb{x},\mathbb{y}), & \text{if $\mathbb{x}\equiv x,\mathbb{y}\equiv y$ belong to the same copy of $X$}, \\ +\infty, & \text{else},
\end{cases}
\end{equation}
where we write $\mathbb{x} \equiv x$ if we identify $\mathbb{x} \in \mathbb{X}$ with an element $x \in X$ (belonging to either the first or else the second copy of $X$).

Moreover, we  identify the function space $X^{\CC}\times X^{\CC}\simeq X^{2\CC}$ with $\mathbb{X}^\CC$ by 
\((f,g) \mapsto f \sqcup g\), where
\[
f \sqcup g
 : \mathbb{X} \ni \mathbb{x} \mapsto \begin{cases} f(x), & \text{if }\mathbb{x}=x \in X \sqcup \emptyset, \\ g(x), & \text{if } \mathbb{x}=x \in \emptyset \sqcup X. \end{cases}
\]
Now, let us remark the following: let us endow $C^{0,\alpha}(X) \times C^{0,\alpha}(X)$ with the canonical norm given by
\[
\| (f,g) \|_{C^{0,\alpha} \times C^{0,\alpha}} := \| f \|_{C^{0,\alpha}} + \| g \|_{C^{0,\alpha}},\qquad f,g\in C^{0,\alpha}(X).
\]
If $x,y \in X$ belong to the same copy of $X$, we have that $d_{\mathbb X}(\mathbb{x},\mathbb{y}) = d(x,y)$ for $\mathbb{x}\equiv x$ and $\mathbb{y}\equiv y$, and thus
\begin{align*}
\frac{\vert (f \sqcup g)(x) - (f \sqcup g)(y) \vert}{d_{\mathbb X}(x,y)^\alpha} = \begin{cases} \frac{\vert f(x) - f(y) \vert}{d(x,y)^\alpha}, & \text{if }x \in X \sqcup \emptyset, \\ \frac{\vert g(x) - g(y) \vert}{d(x,y)^\alpha}, & \text{if }x \in \emptyset \sqcup X, \end{cases}
\end{align*}
whence
\begin{align*}
\frac{\vert (f \sqcup g)(\mathbb{x}) - (f \sqcup g)(\mathbb{y}) \vert}{d_{\mathbb X}(\mathbb{x},\mathbb{y})^\alpha} \leq \vert f \vert_{C^{0,\alpha}} + \vert g \vert_{C^{0,\alpha}}.
\end{align*}
Because this estimate holds trivially whenever $d_{\mathbb X}(\mathbb{x},\mathbb{y}) = \infty$, we conclude that
\[
\vert f \sqcup g \vert_{C^{0,\alpha}(\mathbb X)} \leq \vert (f,g) \vert_{C^{0,\alpha} \times C^{0,\alpha}}\qquad\hbox{for all }
f,g \in C^{0,\alpha}(X).
\]
Conversely, 
\[
\big\vert f_{|X \sqcup \emptyset} \big\vert_{C^{0,\alpha}} + \big\vert f_{|\emptyset \sqcup X} \big\vert_{C^{0,\alpha}} \leq 2\vert f \vert_{C^{0,\alpha}(\mathbb X)} \qquad \text{for all $f \in C^{0,\alpha}(\mathbb{X})$.}
\] 
We conclude that, for each $\alpha \in (0,1]$,
\begin{equation*}
C^{0,\alpha}(X) \times C^{0,\alpha}(X)\simeq C^{0,\alpha}(\mathbb X).
\end{equation*}
\end{rem}

\noindent
\textbf{Data availability statement}
No new data were created or analyzed in this study.


\begin{thebibliography}{10}

\bibitem{AdaFou03}
R.A. Adams and J.J.F. Fournier.
\newblock {\em {Sobolev Spaces}}.
\newblock Elsevier, Amsterdam, 2003.

\bibitem{AreBuk94}
W.\ Arendt and A.V.\ Bukhvalov.
\newblock Integral representations of resolvents and semigroups.
\newblock {\em Forum Math.}, 6:111--135, 1994.

\bibitem{AreDieLaa14}
W.~Arendt, D.~Dier, H.~Laasri, and E.M. Ouhabaz.
\newblock Maximal regularity for evolution equations governed by non-autonomous
  forms.
\newblock {\em Adv.\ Diff.\ Equ.}, 19:1043--1066, 2014.

\bibitem{AreEls97}
W.\ Arendt and A.F.M. ter Elst.
\newblock Gaussian estimates for second order elliptic operators with boundary
  conditions.
\newblock {\em J.\ Operator Th.}, 38:87--130, 1997.

\bibitem{AreEls19}
W.~Arendt and A.F.M. ter Elst.
\newblock Operators with continuous kernels.
\newblock {\em Integr. Equ. Oper. Theory}, 91:45, 2019.

\bibitem{AusCouDuo04}
P.~Auscher, T.~Coulhon, X.T. Duong, and S.~Hofmann.
\newblock Riesz transform on manifolds and heat kernel regularity.
\newblock {\em Ann.\ Sci.\ École Norm.\ Sup.}, 37:911--957, 2004.

\bibitem{AusMcITch98}
P.~Auscher, A.~McIntosh, and P.~Tchamitchian.
\newblock Heat kernels of second order complex elliptic operators and
  applications.
\newblock {\em J.\ Funct.\ Anal.}, 152:22--73, 1998.

\bibitem{Bat04}
C.J.K. Batty.
\newblock Differentiability and growth bounds of solutions of delay equations.
\newblock {\em J.\ Math.\ Anal.\ Appl.}, 299:133--146, 2004.

\bibitem{BecGreMug21}
S.~Becker, F.~Gregorio, and D.~Mugnolo.
\newblock {S}chrödinger and polyharmonic operators on infinite graphs:
  Parabolic well-posedness and $p$-independence of spectra.
\newblock {\em J.\ Math.\ Anal.\ Appl.}, 495:124748, 2021.


\bibitem{BjoKal22}
J.~Bj{\"o}rn and A.~Ka{\l}amajska.
\newblock Poincar{\'e} inequalities and compact embeddings from {Sobolev} type
  spaces into weighted {{\(L^q\)}} spaces on metric spaces.
\newblock {\em J.\ Funct.\ Anal.}, 282:47, 2022.

\bibitem{CavMon17}
F.~Cavalletti and A.~Mondino.
\newblock {Sharp geometric and functional inequalities in metric measure spaces
  with lower Ricci curvature bounds}.
\newblock {\em Geometry \& Topology}, 21:603--645, 2017.

\bibitem{Che99}
J.\ Cheeger.
\newblock Differentiability of {L}ipschitz functions on metric measure spaces.
\newblock {\em GAFA, Geom.\ Funct.\ Anal.}, 9:428--517, 1999.

\bibitem{CheTri89}
S.P. Chen and R.~Triggiani.
\newblock Proof of extensions of two conjectures on structural damping for
  elastic systems.
\newblock {\em Pacific J.\ Math.}, 136:15--55, 1989.

\bibitem{CheTri90}
S.P. Chen and R.~Triggiani.
\newblock Characterization of domains of fractional powers of certain operators
  arising in elastic systems, and applications.
\newblock {\em J.\ Differ.\ Equ.}, 88:278--293, 1990.

\bibitem{Con85}
J.B. Conway.
\newblock {\em {A Course In Functional Analysis}}, volume~96 of {\em Graduate
  Texts in Mathematics}.
\newblock Springer-Verlag, Berlin, 1985.

\bibitem{Cou03}
T.~Coulhon.
\newblock Off-diagonal heat kernel lower bounds without {P}oincar{\'e}.
\newblock {\em J.\ London Math.\ Soc.}, 68:795--816, 2003.

\bibitem{CouJiaKos20}
T.~Coulhon, R.~Jiang, P.~Koskela, and A.~Sikora.
\newblock Gradient estimates for heat kernels and harmonic functions.
\newblock {\em J.\ Funct.\ Anal.}, 278:108398, 2020.

\bibitem{CreSou20}
P.~Creutz and E.~Soultanis.
\newblock {Maximal metric surfaces and the Sobolev-to-Lipschitz property}.
\newblock {\em Calc.\ Var.}, 59:1--34, 2020.

\bibitem{DauLio88}
R.\ Dautray and J.-L.\ Lions.
\newblock {\em {Mathematical Analysis and Numerical Methods for Science and
  Technology, Vol.\ 2}}.
\newblock Springer-Verlag, Berlin, 1988.

\bibitem{Dav89}
E.B. Davies.
\newblock {\em Heat {K}ernels and {S}pectral {T}heory}, volume~92 of {\em
  Cambridge Tracts Math.}
\newblock Cambridge Univ.\ Press, Cambridge, 1989.

\bibitem{DelSuz22}
L.~Dello~Schiavo and K.~Suzuki.
\newblock {Sobolev-to-Lipschitz property on QCD-spaces and applications}.
\newblock {\em Math.\ Ann.}, 384:1815--1832, 2022.

\bibitem{EgiMugSee23}
M.~Egidi, D.~Mugnolo, and A.~Seelmann.
\newblock Sturm-Liouville problems and global bounds by small control sets and applications to quantum graphs.
\newblock {\em J.\ Math.\ Anal.\ Appl.},
DOI: j.jmaa.2024.128101, 2024.

\bibitem{ElsOuh19}
{A.F.M.\ ter} Elst and E.M. Ouhabaz.
\newblock {Dirichlet-to-Neumann and elliptic operators on $C^{1+
  \kappa}$-domains: Poisson and Gaussian bounds}.
\newblock {\em J.\ Differ.\ Equ.}, 267:4224--4273, 2019.

\bibitem{ElsReh15}
{A.F.M.\ ter} Elst and J.~Rehberg.
\newblock Hölder estimates for second-order operators on domains with rough
  boundary.
\newblock {\em Adv.\ Diff.\ Equ.}, 20:299--360, 2015.

\bibitem{ElsRob08}
{A.F.M.\ ter} Elst and D.W. Robinson.
\newblock Contraction semigroups on {$L_\infty(R)$}.
\newblock In H.~Amann, W.~Arendt, M.~Hieber, F.~Neubrander, S.~Nicaise, and {J.
  von} Below, editors, {\em Functional {A}nalysis and {E}volution {E}quations
  -- {T}he {G}ünter {L}umer {V}olume}, pages 503--514. Birkh{\"a}user, Basel,
  2008.

\bibitem{ElsWon20}
{A.F.M.\ ter} Elst and M.F. Wong.
\newblock {H{\"o}lder kernel estimates for Robin operators and
  Dirichlet-to-Neumann operators}.
\newblock {\em J.\ Evol.\ Equ.}, 20:1195--1225, 2020.

\bibitem{EngNag00}
K.-J.\ Engel and R.\ Nagel.
\newblock {\em One-{P}arameter {S}emigroups for {L}inear {E}volution
  {E}quations}, volume 194 of {\em Graduate Texts in Mathematics}.
\newblock Springer-Verlag, New York, 2000.

\bibitem{Eva10}
L.C.\ Evans.
\newblock {\em Partial {D}ifferential {E}quations -- second edition}, volume~19
  of {\em Graduate Studies in Mathematics}.
\newblock Amer.\ Math.\ Soc., Providence, RI, 2010.

\bibitem{FreRui17}
U.R.\ Freiberg and P.A.\ Ruiz.
\newblock {\em Weyl asymptotics for Hanoi attractors}.
\newblock Forum Mathematicum, 20:1003--1021, 2017.

\bibitem{FreKigRui18}
U.R.\ Freiberg, J.\ Kigami and P.A.\ Ruiz.
\newblock {\em Completely symmetric resistance forms on the stretched Sierpiński gasket}.
\newblock . Fractal Geom.\ 5, 3:227--277, 2018.

\bibitem{Gig13}
N.~Gigli.
\newblock The splitting theorem in non-smooth context.
\newblock arXiv:1302.5555, 2013.

\bibitem{Gig15}
N.~Gigli.
\newblock {\em On the differential structure of metric measure spaces and
  applications}, volume 236 of {\em Mem.\ Amer.\ Math.\ Soc.}
\newblock Amer.\ Math.\ Soc., Providence, RI, 2015.

\bibitem{GigHan18}
N.~Gigli and B.-X. Han.
\newblock Sobolev spaces on warped products.
\newblock {\em J.\ Funct.\ Anal.}, 275:2059--2095, 2018.

\bibitem{GriTel12}
A.~Grigor'yan and A.~Telcs.
\newblock Two-sided estimates of heat kernels on metric measure spaces.
\newblock {\em Ann.\ Probab.}, 40:1212--1284, 2012.

\bibitem{Hei05}
J.~Heinonen.
\newblock {Lectures on Lipschitz Analysis}.
\newblock \url{http://www.math.jyu.fi/research/reports/rep100.pdf}, 2005.

\bibitem{Hor90}
L.~Hörmander.
\newblock {\em {The Analysis of Linear Partial Differential Operators I}},
  volume 256 of {\em Grundlehren der mathematischen Wissenschaften}.
\newblock Springer-Verlag, Berlin, 1990.

\bibitem{Kig02}
J.\ Kigami.
\newblock Harmonic analysis for resistance forms.
\newblock {\em J.\ Func.\ Anal.}, 204:399--444, 2002.

\bibitem{KosMugNic22}
A.\ Kostenko, D.\ Mugnolo, and N.\ Nicolussi.
\newblock Self-adjoint and {M}arkovian extensions of infinite quantum graphs.
\newblock {\em J.\ London Math.\ Soc.}, 105:1262--1313, 2022.

\bibitem{Laa18}
H.\ Laasri.
\newblock Regularity properties for evolution families governed by
  non-autonomous forms.
\newblock {\em Arch.\ Math.}, 111:187--201, 2018.

\bibitem{LaaMug20}
H.\ Laasri and D.\ Mugnolo.
\newblock Ultracontractivity and {G}aussian estimates for evolution families
  associated with non-autonomous forms.
\newblock {\em Math.\ Meth.\ Appl.\ Sci.}, 43:1409--1436, 2020.

\bibitem{Lun95}
A.\ Lunardi.
\newblock {\em Analytic {S}emigroups and {O}ptimal {R}egularity in {P}arabolic
  {P}roblems}, volume~16 of {\em Progress in Nonlinear Differential Equations
  and their Applications}.
\newblock Birkh{\"a}user, Basel, 1995.

\bibitem{Mug19}
D.~Mugnolo.
\newblock What is actually a metric graph?
\newblock arXiv:1912.07549.

\bibitem{Mug14}
D.~Mugnolo.
\newblock {\em {Semigroup Methods for Evolution Equations on Networks}}.
\newblock Underst.\ Compl.\ Syst. Springer-Verlag, Berlin, 2014.

\bibitem{MugPlu23}
D.\ Mugnolo and M.\ Pl{\"u}mer.
\newblock On torsional rigidity and ground-state energy of compact quantum
  graphs.
\newblock {\em Calc.\ Var.}, 62:27, 2023.

\bibitem{Ouh98}
E.M. Ouhabaz.
\newblock Heat kernels of multiplicative perturbations: H{\"o}lder estimates
  and gaussian lower bounds.
\newblock {\em Indiana Univ.\ Math.\ J.}, 47:1481--1495, 1998.

\bibitem{Ouh05}
E.M. Ouhabaz.
\newblock {\em Analysis of {H}eat {E}quations on {D}omains}, volume~30 of {\em
  Lond.\ Math.\ Soc.\ Monograph Series}.
\newblock Princeton Univ.\ Press, Princeton, NJ, 2005.

\bibitem{PosRuc18}
O.~Post and R.~R{\"u}ckriemen.
\newblock Locality of the heat kernel on metric measure spaces.
\newblock {\em Compl.\ Anal.\ Oper.\ Theory}, 12:729--766, 2018.

\bibitem{Pro15}
A.~Profeta.
\newblock {The sharp Sobolev inequality on metric measure spaces with lower
  Ricci curvature bounds}.
\newblock {\em Potential Analysis}, 43:513--529, 2015.

\bibitem{Ren95}
M.~Renardy.
\newblock On the stability of differentiability of semigroups.
\newblock {\em Semigroup Forum}, 51:343--346, 1995.

\bibitem{Rui13}
P.A.~Ruiz.
\newblock Dirichlet forms on non self-similar sets:
Hanoi attractors and the Sierpiński gasket.
\newblock {\em Dissertation thesis}, 2013.

\bibitem{Sim82}
B.~Simon.
\newblock Schr{\"o}dinger semigroups.
\newblock {\em Bull.\ Amer.\ Math.\ Soc.}, 7:447--526, 1982.

\bibitem{Str06}
{A.S.} Strichartz.
\newblock {Differential Equations on Fractals: A Tutorial}.
\newblock {\em Princeton University Press}, 2006.

\bibitem{Tan79}
H.~Tanabe.
\newblock {\em {Equations Of Evolution}}, volume~6 of {\em Monogr.\ Studies
  Math.}
\newblock Pitman, Boston, 1979.

\bibitem{XiaLia98}
T.-J.\ Xiao and J.\ Liang.
\newblock {\em {The Cauchy Problem for Higher-Order Abstract Differential
  Equations}}, volume 1701 of {\em Lect.\ Notes Math.}
\newblock Springer-Verlag, Berlin, 1998.

\end{thebibliography}

\end{document}